\documentclass[11pt]{amsart}

\usepackage{amssymb}
\usepackage{amsmath}
\usepackage{amsthm}
\usepackage{graphics}

\allowdisplaybreaks

\numberwithin{equation}{section}

\theoremstyle{plain}
\newtheorem{thm}{Theorem}[section]
\newtheorem{prop}[thm]{Proposition}
\newtheorem{lem}[thm]{Lemma}

\theoremstyle{definition}

\newtheorem{rem}[thm]{Remark}

\newcommand{\set}[1]{\left\{#1\right\}}%
\newcommand{\abs}[1]{\left\vert#1\right\vert}%
\newcommand{\norm}[1]{\left\Vert{#1}\right\Vert}%
\newcommand{\one}{\mathrm{I}}%
\newcommand{\two}{\mathrm{II}}%
\newcommand{\three}{\mathrm{III}}%
\DeclareMathOperator{\supp}{\mathrm{supp}\,}

\begin{document}


\title[Multilinear Fourier Multipliers with Minimal Regularity, II]{Multilinear Fourier Multipliers with Minimal Sobolev Regularity, II}

\author[Grafakos]{Loukas Grafakos}
\address{Department of Mathematics, University of Missouri, Columbia, MO 65211, USA}
\email{grafakosl@missouri.edu}
\thanks{The first author would like to thank the Simons Foundation and the University of Missouri Research Council.}

\author[Miyachi]{Akihiko Miyachi}

\address{Akihiko Miyachi \\
Department of Mathematics \\
Tokyo Woman's Christian University \\
Zempukuji, Suginami-ku, Tokyo 167-8585, Japan}
\email{miyachi@lab.twcu.ac.jp}

\author[Nguyen]{Hanh Van Nguyen}
\address{Department of Mathematics, University of Missouri, Columbia, MO 65211, USA}
\email{hnc5b@mail.missouri.edu}
\thanks{The third author would like to thank Hue University - College of Education for their support.}

\author[Tomita]{Naohito Tomita}

\address{Naohito Tomita \\
Department of Mathematics \\
Osaka University \\
Toyonaka, Osaka 560-0043, Japan}
\email{tomita@math.sci.osaka-u.ac.jp}

\begin{abstract}
We  provide  characterizations for  boundedness of
multilinear Fourier   operators on   Hardy or Lebesgue spaces with symbols locally in Sobolev spaces.
Let $H^q(\mathbb R^n)$
 denote      the Hardy space when $0<q\le 1$ and the Lebesgue space $L^q(\mathbb R^n)$ when $1<q\le \infty$.
We find optimal conditions
 on $m$-linear Fourier multiplier operators to be bounded from  $H^{p_1}\times \cdots \times H^{p_m}$ to $L^p$ when $1/p=1/p_1+\cdots +1/p_m$  in terms of local $L^2$-Sobolev space estimates for the symbol of the operator.
Our conditions provide multilinear analogues of the linear   results  of
 Calder\'on and Torchinsky   \cite{CalTochin}
 and of   the bilinear results
of  Miyachi and Tomita \cite{MiTo}. The extension to general $m$ is
significantly more complicated  both technically and combinatorially; the optimal
Sobolev space smoothness   required of the symbol depends on the Hardy-Lebesgue
exponents and is constant on various convex simplices
formed by configurations of    $m2^{m-1} +1$ points  in $[0,\infty)^m$.
\end{abstract}

\maketitle

\section{Introduction}\label{section:Introduction}
We denote by $T_{\sigma}$ the linear Fourier multiplier operator, acting on
Schwartz functions $f$, defined  by
\begin{equation}\label{equ:LiMulOp}
T_{\sigma}(f)(x) = \int_{\mathbb R^{n}} {\sigma}(\xi) \widehat{f}(\xi)e^{2\pi ix\cdot \xi}\, d\xi,
\end{equation}
where $\sigma$ is a bounded function on $\mathbb R^{n}$
and $\widehat{f}(\xi)=\int_{\mathbb R^n} f(x) e^{-2\pi i x\cdot \xi} dx$ denotes the Fourier transform of $f$.
 H\"ormander \cite{Horman}   proved that $T_\sigma$ is bounded
 from $L^p(\mathbb R^n)$ to itself for $1<p<\infty$  if
 \begin{equation}\label{equ:gghhtt2}
\sup_{j\in\mathbb Z}
\Big\| {{\sigma}(2^{j} \cdot )
\widehat{\psi}}\,
\Big\| _{W^s}
<\infty
\end{equation}
for some $s>\frac{n}2,$ where $\widehat{\psi}$ is a smooth function supported in
$\frac12\le{\left\vert{\xi}\right\vert}\le 2$ that satisfies
$$
\sum_{j\in\mathbb Z}\widehat\psi(2^{-j}\xi) =1
$$
for all $\xi\ne 0.$ In this paper,  $W^s$ denotes the Sobolev space with norm
$$
\|g\|_{W^s} =\| (I-\Delta)^{s/2}g\|_{L^2},
$$
 where $I$ is the identity operator and $\Delta= \sum_{j=1}^n \partial_j^2$ is the Laplacian on $\mathbb R^n$.
 H\"ormander's result strengthens an earlier result
of Mikhlin \cite{Miklin}.

Throughout  this work, $H^p(\mathbb R^{n})$ denotes the real-variable Hardy space of Fefferman and Stein \cite{FS},  for $0<p \le \infty$.
This space  coincides with the Lebesgue space  $L^p(\mathbb R^{n})$ when $1<p\le \infty$.
Calder\'on and Torchinsky  \cite{CalTochin}  provided an extension  of H\"ormander's result   to $H^p(\mathbb R^{n})$ for $p\le 1$. They showed  that the   Fourier multiplier operator in \eqref{equ:LiMulOp}
admits a bounded extension from the Hardy space  $H^p(\mathbb R^{n})$ to $H^p(\mathbb R^{n})$ with  $0<p\le 1$ if
$$
\sup_{t>0}\left\Vert{{\sigma}(t\cdot)\widehat{\psi}}\, \right\Vert_{W^s}<\infty
$$
and $s>\frac{n}{p} - \frac{n}{2} $. Moreover,
the boundedness of $T_\sigma$ on $H^p$ may not hold if
$s< \frac{n}{p} - \frac{n}{2}$; in other words, the  Calder\'on and Torchinsky condition $s>\frac{n}{p} - \frac{n}{2} $ is sharp (for this, see for instance \cite[Remark 1.3]{MiTo}).

In this work we study   analogues of these results for multilinear multipliers defined on
products of Hardy or Lebesgue spaces on the entire range of indices $0<p\le \infty$.  Multilinear multiplier operators  were studied by  Coifman and Meyer
\cite{CM1}, \cite{CM2}, \cite{CoiMey} and more recently by
Grafakos and Torres \cite{GraTor}.
Multilinear Fourier multiplier is a bounded function
 $\sigma$ on $\mathbb R^{mn} = \mathbb R^n\times\cdots\times \mathbb R^n$
  associated with the $m$-linear Fourier multiplier operator
\begin{equation}\label{equ:TSigmaMul}
T_{\sigma}(f_1,\ldots,f_m)(x) = \int_{\mathbb R^{mn}}e^{2\pi ix\cdot (\xi_1+\cdots+\xi_m)}\sigma(\xi_1,\ldots,\xi_m)\widehat{f_1}(\xi_1)\cdots \widehat{f_m}(\xi_m)\, d\vec \xi ,
\end{equation}
where  $f_j$ are in the Schwartz space of $\mathbb R^n$ and $d\vec \xi =d\xi_1\cdots d\xi_m  $.

A short history of the known results concerning multilinear multipliers with minimal smoothness is as follows:
Tomita \cite{TomitaJFA}
obtained $L^{p_1}\times \cdots \times L^{p_m}\to L^p$   boundedness ($1<p_1,\ldots,p_m,p< \infty$) for   multilinear multiplier operators
under a condition     \eqref{equ:gghhtt2}.  Grafakos and Si \cite{GraSi} extended
Tomita's result to the case $p\le 1$   using   $L^r$-based Sobolev spaces with $1<r\le 2.$
Fujita and Tomita \cite{FuTomi} provided weighted  extensions of these results and also noticed that the Sobolev space $W^s$ in \eqref{equ:gghhtt2} can be
replaced by a product-type Sobolev space $W^{(s_1,\dots ,s_m)}$ when $p> 2$.
Grafakos, Miyachi, and Tomita \cite{GrMiTo} extended the range of $p$ in \cite{FuTomi}
to $p>0$ and obtained the boundedness even  in the endpoint case
where all but one indices $p_j$ are equal to infinity.
Miyachi and Tomita \cite{MiTo} provided  sharp conditions on the entire range of  indices ($0<p_j\le \infty$), extending    the Calder\'on and Torchinsky   \cite{CalTochin}  result   to the case $m=2$.

In this work we provide extensions of the result of
Calder\'on and Torchinsky \cite{CalTochin}  ($m=1$) and of
Miyachi and Tomita \cite{MiTo}  ($m=2$)
to the cases $m\ge 3$. We point out that the complexity of the problem increases significantly as $m$ increases. In fact, the main difficulty concerns the case where $1<p_j<2$, in which the
boundedness holds exactly in the interior of a convex simplex in $\mathbb R^m$. This simplex has
 $m2^{m-1}+1$ vertices   but  it is not enough to obtain the corresponding estimates for the vertices of the simplex, since interpolation between the vertices does not yield minimal smoothness in the interior. We overcome this difficulty by
 establishing
 estimates for all the points inside the simplex being arbitrarily close to those $m2^{m-1}+1$ points without losing   smoothness.

Before stating our  main result  we introduce some notation.
First, for $x\in \mathbb R^n$ we set $\left<x\right> = \sqrt{1+ |x| ^2}.$
For $s_1,\ldots,s_m>0,$ we denote by $W^{(s_1,\ldots,s_m)}$ the product-type-Sobolev space consisting of
all functions    $f$ on $\mathbb R^{mn}$ such that
\begin{equation*}
\left\Vert{f}\right\Vert_{W^{(s_1,\ldots,s_m)}}:=
\left(
\int_{\mathbb R^{mn}}
\left\vert \widehat{f}(y_1,\ldots,y_m)
\left<y_1\right>^{s_1}\cdots \left<y_m\right>^{s_m}
\right\vert^2 dy_1\cdots dy_m
\right)^{\frac12} <\infty.
\end{equation*}
Notice that $W^{(s_1,\ldots,s_m)}$ is a subspace of $L^2$.

We denote by $\psi$ a smooth function on $\mathbb R^{mn}$ whose Fourier transform $\widehat{\psi}$ is supported in $\frac12\le {\left\vert{\xi}\right\vert}\le 2$ and satisfies
\begin{equation*}
\sum_{j\in \mathbb Z}\widehat{\psi}(2^{-j}\xi) = 1,\qquad \xi\ne 0.
\end{equation*}

The following is the main result of this paper. It concerns boundedness of operators
of the form \eqref{equ:TSigmaMul} on products of  Hardy spaces in the full range of indices.

\begin{thm}\label{thm:General}
Let $0<p_1,\ldots,p_m\le \infty$, $0< p< \infty$,
$\frac{1}{p_1}+ \cdots + \frac{1}{p_m} = \frac{1}{p}$,
$s_1,\ldots,s_m> n/2$, and suppose
  \begin{equation}\label{equ:SIndexCond}
  \sum_{k\in J}\Big(\dfrac{s_k}{n}-\dfrac1{p_k}\Big)>-\dfrac12
  \end{equation}
for every nonempty subset $J\subset\left\{1,2,\ldots,m\right\}$.
If $\sigma$ satisfies
\begin{equation}\label{equ:SigmCond}
A:=\sup_{j\in\mathbb Z}
\left\Vert{\sigma(2^j\cdot)\widehat\psi}\,
\right\Vert_{W^{(s_1,\ldots,s_m)}}<\infty,
\end{equation}
then we have
\begin{equation}\label{equ:TSigmBound}
\norm{T_{\sigma}}_{H^{p_1}\times\cdots\times H^{p_m}
\longrightarrow  L^p}
\lesssim
A.
\end{equation}
Moreover, this result is optimal
in the sense that if \eqref{equ:SigmCond} and \eqref{equ:TSigmBound}
are valid
then
we must necessarily have
$s_1, \dots, s_m \ge n/2$ and
  \begin{equation}\label{equ:SIndexCondGe}
  \sum_{k\in J}\Big(\dfrac{s_k}{n}-\dfrac1{p_k}\Big)\ge -\dfrac12
  \end{equation}
for every nonempty subset $J\subset\left\{1,2,\ldots,m\right\}$.
\end{thm}

\begin{rem}\label{rem:aaa}

This paper is a sequel of   \cite{GraHan} for the following reasons:
\begin{enumerate}

\item The case $p_i\le 1$ for all $1\le i\le m$ is contained in \cite{GraHan}.

\item The endpoint case of Theorem \ref{thm:General} in the case where $p_i=p=\infty$  for all $i\in \{1,\dots , m\}$ is   proved in \cite{GraHan}:
$$
\norm{T_{\sigma}(f_1,\ldots,f_m)}_{BMO}\lesssim \sup_{j\in \mathbb Z}\norm{\sigma(2^j\cdot)\widehat{\psi}\, }  _{W^{(s_1,\ldots,s_m)}}
\prod_{i=1}^m\norm{f_i}_{L^\infty}
$$
for $s_1, \dots, s_m>n/2$.

\item The necessity of the conditions
$s_i \ge n/2$ and \eqref{equ:SIndexCondGe} was shown
in \cite[Theorem 5.1]{GraHan} for the
entire range of indices $0<p_j\le \infty$, $0<p<\infty$.

\end{enumerate}
\end{rem}

We  will consistently  use the notation
$A\lesssim B$ to indicate that $A\le CB$ for
some constant $C>0,$ and $A\approx B$
if $A\lesssim B$ and $B\lesssim A$ simultaneously.

The paper is structured as follows.
Section \ref{section:Preliminaries} contains preliminaries and known results.
In Section \ref{section:ProofMain}, we give the proof of the main result
by considering four cases.
In Section \ref{section:ProofLem}, we present
detailed proofs of the lemmas used in Section \ref{section:ProofMain}.
In the last section, Section \ref{section:WeakEstimate},
we give a result concerning the space $L^1$
and weak type estimate.

\section{Preliminaries and known results}\label{section:Preliminaries}

Now fix $0<p\le \infty$ and a Schwartz function $\Phi$ with $\widehat{\Phi}(0)\ne 0.$ Then the Hardy space $H^p$ contains all tempered distributions $f$ on $\mathbb R^n$ such that
$$
\left\Vert{f}\right\Vert_{H^p}:= \left\Vert{\sup_{0<t<\infty}{\left\vert{\Phi_t*f}\right\vert}}\right\Vert_{L^p}<\infty.
$$
It is well known that the definition of the Hardy space does not depend on the choice of the function $\Phi.$
Note that $H^p=L^p$ for all $p>1.$
When $0<p\le 1,$ one of nice features of Hardy spaces is the atomic decomposition. More precisely, any function $f\in H^p$ ($0<p\le 1$) can be decomposed as
$f=\sum_{k}\lambda_k a_k,$
where $a_k$'s are $L^\infty$-atoms for $H^p$ supported in
 cubes $Q_k$
 such that $\left\Vert{a_k}\right\Vert_{L^{\infty}}\le {\left\vert{Q_k}\right\vert}^{-\frac1p}$
and $\int x^\gamma a_k(x)dx=0$ for all ${\left\vert{\gamma}\right\vert}< N$,
and the coefficients $\lambda_k$ satisfy
$\sum_k{\left\vert{\lambda_k}\right\vert}^p\le 2^p\left\Vert{f}\right\Vert_{H^p}^p$.
The order $N$ of the moment condition can be taken arbitrarily large.

A fundamental $L^2$ estimate for $T_{\sigma}$
is given in the following theorem.
\begin{thm}[\cite{GrMiTo}]\label{thm:L2estimate}
If $s_1,\,\dots,\,s_m >n/2$, then
  $$
  \left\Vert{T_{\sigma}}\right\Vert
  _{
  L^{2}\times L^{\infty} \times \cdots\times L^{\infty}
  \longrightarrow  L^{2}
  }
  \le C \sup_{j\in\mathbb Z}
  \left\Vert{\sigma(2^j\cdot)\widehat\psi}\,
  \right\Vert
  _{W^{(s_{1},\ldots,s_{m})}}.
  $$
\end{thm}

The following two lemmas are essentially contained in \cite{MiTo},
modulo a few minor modifications.
\begin{lem}[\cite{MiTo}]\label{lem:LInfL2}
Let $m$ be a positive integer,
${\sigma}$ be a function defined on
$\mathbb R^{mn}$, and
$K={\sigma}^{\vee},$
the inverse Fourier transform of $\sigma.$
Suppose
$\sigma$ is supported in
$\big\{y \in\mathbb R^{mn}\, :\ \abs{y}\le 2\big\}$
and suppose $s_i\ge 0$ for $1\le i\le m$
and $1\le p\le q\le \infty.$
Then for each multi-index $\alpha$
there exists a constant $C_{\alpha}$ such that
$$
\left\Vert{\left<y_1\right>^{s_{1}} \cdots\left<y_l\right>^{s_{l}}
\partial_{y}^{\alpha}
K(y)}\right\Vert_{L^{q}(\mathbb R^{ml},\; dy_1\cdots dy_l)}
\le C_{\alpha} \left\Vert{\left<y_1\right>^{s_{1}}\cdots
\left<y_l\right>^{s_{l}}K(y)}\right\Vert_{L^{p}(\mathbb R^{ml},\; dy_1\cdots dy_l)},
$$
where $y=(y_1,\ldots,y_m)$ with $y_j\in\mathbb R^n$.
\end{lem}

\begin{lem}[\cite{MiTo}]\label{lem:ProdSobN}
  Let $s_i> \frac{n}{2}$ for $1\le i\le m,$ and let $\widehat{\zeta}$ be
  a smooth function which is supported in an annulus centered at zero.
  Suppose that $\Phi$ is a smooth function away from zero that satisfies the estimates
  $$
  {\left\vert{\partial^{\alpha}_{\xi}\Phi(\xi)}\right\vert}\le
  C_{\alpha}{\left\vert{\xi}\right\vert}^{-{\left\vert{\alpha}\right\vert}}
  $$
  for all $\xi\in \mathbb R^{mn},$ $\xi \ne 0$,
  and for all multi-indices $\alpha.$
  Then there exists a constant $C$ such that
  $$
  \sup_{j\in\mathbb Z}
  \left\Vert{\sigma(2^{j} \cdot )
  \Phi(2^j \cdot )\widehat{\zeta}}\,
  \right\Vert_{W^{(s_1,\ldots,s_m)}}
  \le C\sup_{j\in\mathbb Z}
  \left\Vert{\sigma(2^{j} \cdot )
  \widehat\psi}\, \right\Vert_{W^{(s_1,\ldots,s_m)}}.
  $$
\end{lem}

Adapting the Calder\'on and Torchinsky  interpolation techniques in the multilinear setting (for  details on this we refer  to \cite[p.\,318]{GrMiTo}) allows us to interpolate between two
 estimates for multilinear multiplier operators from a product of some Hardy spaces or Lebesgue spaces to Lebesgue spaces.

\begin{thm}[\cite{GrMiTo}]\label{thm:CalTorInterp}
  Let $0<p_i,p_{i,k}\le \infty$ and $s_{i,k}>0$ for $i=1,2$ and $1\le k\le m.$
  For $0<\theta<1$, set
  $\frac1{p}= \frac{1-\theta}{p_1}+\frac{\theta}{p_2},$
  $\frac1{p_k}
  = \frac{1-\theta}{p_{1,k}}+\frac{\theta}{p_{2,k}},$ and
  $s_k = (1-\theta)s_{1,k}+\theta s_{2,k}.$
  Assume that the multilinear operator $T_{\sigma}$ defined in
  \eqref{equ:TSigmaMul} satisfies the estimates
  $$
  \left\Vert{T_{\sigma}}\right\Vert_{H^{p_{i,1}}\times\cdots\times
  H^{p_{i,m}}\longrightarrow  L^{p_i}}
  \le C_i\sup_{j\in\mathbb Z}
  \left\Vert{\sigma(2^j\cdot)\widehat\psi}\, \right\Vert
  _{W^{(s_{i,1},\ldots,s_{i,m})}},\qquad i=1,2,
  $$
  where $L^{p_i}$ should be replaced by $BMO$ if $p_i=\infty$.
  Then
  $$
  \left\Vert{T_{\sigma}}\right\Vert_{H^{p_{1}}\times\cdots\times
  H^{p_{m}}\longrightarrow  L^{p}}
  \le C\sup_{j\in\mathbb Z}\left\Vert{\sigma(2^j\cdot)
  \widehat\psi}\,\right\Vert_{W^{(s_{1},\ldots,s_{m})}},
  $$
    where $L^{p}$ should be replaced by $BMO$ if $p=\infty$.
\end{thm}

Fix a Schwartz function $K.$
We denote the multilinear
operator of convolution type associated with the kernel $K$ by
$$
{ T^K}(f_1,\ldots,f_m)(x) = \int_{\mathbb R^{mn}}
K(x-y_1,\ldots,x-y_m)f_1(y_1)\cdots f_m(y_m)dy_1\cdots dy_m.
$$
The following result can be verified with a very similar argument
as showed in \cite[Proposition 3.4]{GraHan}.
\begin{prop}\label{pro:SwitchSum}
  Let $0<p_1, \dots, p_l\le 1$ and
$1<p_{l+1}, \dots, p_m\le \infty$.
Let $K$ be a smooth function with compact support.
Suppose $f_i\in H^{p_i}$, $1\le i\le l$,
has atomic representation
$f_i = \sum_{k_i}\lambda_{i,k_i}a_{i,k_i},$
where $a_{i,k_i}$ are $L^{\infty}$-atoms for $H^{p_i}$ and
$\sum_{k_i}{\left\vert{\lambda_{i,k_i}}\right\vert}^{p_i}\le
2^{p_i}\norm{f_i}^{p_i}_{H^{p_i}}$.
Suppose $f_i \in L^{p_i}$ for $l + 1\le i \le m$.
Then
  $$
 {  T^K}(f_1,\ldots,f_m)(x) =
\sum_{k_1}\cdots\sum_{k_l}\lambda_{1,k_1}\cdots\lambda_{l,k_l} {  T^K}(a_{1,k_1},\ldots,a_{l,k_l},f_{l+1},\ldots,f_m)(x)
  $$
for almost all $x\in\mathbb R^n.$
\end{prop}

We also use the following lemmas.

\begin{lem}[{\cite[Lemma 2.1]{GrKal}}]\label{lem:LpAvg}
  Let $0<p\le 1$ and let $\Big(f_Q\Big)_{Q\in \mathcal J}$
be a family of nonnegative integrable functions with
$\supp(f_Q)\subset Q$ for all $Q\in\mathcal{J},$
where $\mathcal{J}$ is a family of finite or
countable cubes in $\mathbb R^n.$
Then we have
  $$
\left\Vert
{\sum_{Q\in\mathcal{J}} f_Q}
\right\Vert_{L^p}
\lesssim
\left\Vert
{
\sum_{Q\in\mathcal{J}}
\Big(\frac1{{\left\vert{Q}\right\vert}}
\int_Q f_Q(x)\, dx\Big)
\chi_{Q}
}
\right\Vert_{L^p},
  $$
  where the constant of the inequality depends only on $p.$
\end{lem}

\begin{lem}[{\cite[Lemma 3.3]{GrMiTo}}]\label{lem:Tomita}
Let $s>n/2$, $\max\{1,n/s\}<q<2$, and
$$
\zeta_{j}(x) = 2^{jn} (1+\abs{2^{j}x})^{-sq},
\quad j\in \mathbb Z, \quad x\in {\mathbb R}^{n}.
$$
Suppose $\sigma \in W^{(s,\dots,s)}({\mathbb R}^{mn})$
and
$\mathrm{supp}\, \sigma \subset \{|\xi| \le 2^{j+1}\}$
with a $j\in \mathbb Z$.
Then there exists a constant $C>0$ depending only on
$m$, $n$, $s$, and $q$ such that
\[
|T_{\sigma}(f_1,\dots,f_m)(x)|
\le
C\|\sigma (2^{j}\cdot )\|_{W^{(s,\dots,s)}}
(\zeta_{j}\ast \abs{f_1}^{q})(x)^{1/q}
\dots
(\zeta_{j}\ast \abs{f_m}^{q})(x)^{1/q}
\]
for all $x \in {\mathbb R}^n$.
\end{lem}

\begin{lem}[{\cite[Lemma 3.2]{GrMiTo}}]\label{lem:CarlesonMeasure}
Let $\varphi \in \mathcal{S}({\mathbb R}^n)$ be such that
$\varphi (0)=0$, and set
\begin{equation}\label{equ:Deltaj}
\Delta_{j}f (x)= \int _{\mathbb R^n}
e^{2\pi i x\cdot \xi}\,
\varphi (2^{-j} \xi)
\widehat f (\xi)\,
d\xi,
\quad
j\in \mathbb Z.
\end{equation}
Let $\epsilon >0$ and
$\zeta_{j}(x) = 2^{jn} (1+\abs{2^{j}x})^{-n-\epsilon}$,
$j\in \mathbb Z$, $x\in {\mathbb R}^{n}$.
Then the following inequalities hold
for each $0<q<2$:
\begin{align}
&\sum_{j\in \mathbb Z}
\int_{\mathbb R ^{n}}
|\Delta_{j} f(x)|^2\,
dx
\le C\|f\|_{L^2}^2,
\label{equ:aaa}
\\
&
\sum_{j\in \mathbb Z}
\int_{\mathbb R ^{n}}
(\zeta_{j}\ast |f|^q)(x)^{2/q}
(\zeta_{j}\ast |\Delta_{j} g|^q)(x)^{2/q}
\,
dx
\le C_{q}\|f\|_{L^2}^2\|g\|_{BMO}^2.
\label{equ:ccc}
\end{align}
\end{lem}

\begin{lem}\label{lem:LittlewoodPaley}
Suppose $\{F_j\}\subset \mathcal{S}^{\prime}(\mathbb{R}^n)$
and suppose
there exists a constant $B>1$ such that
$\mathrm{supp}\, \widehat{F_j} \subset
\{\zeta \in \mathbb{R} ^n \mid  B^{-1}2^{j} \le \abs{\zeta}
\le B2^j\}$.
Then, for each $0<p<\infty$,
\[
\left\Vert\sum_{j}F_j\right\Vert_{H^p}
\lesssim \left\Vert\bigg(\sum_{j}|F_j|^2\bigg)^{1/2}
\right\Vert_{L^p}.
\]
\end{lem}

The preceding lemma is well known in the Littlewood-Paley theory,
see for example \cite[5.2.4]{Triebel} and \cite[Lemma 7.5.2]{Grafk}.


\section{The proof of the main result}\label{section:ProofMain}

In this section, we prove the main theorem by considering four cases.
\subsection{The first case: $0<p_i\le 1$, $1\le i\le m$}\label{subsection:FirstCase}
This case is a consequence of  the following result   established in \cite{GraHan}:
\begin{thm}[\cite{GraHan}]\label{thm:MainSmallInd}
  Let $\frac{n}2<s_1,\ldots,s_m<\infty,$ $0<p_1,\ldots,p_m\le 1$,
and $\frac{1}{p_1}+ \cdots + \frac{1}{p_m} = \frac{1}{p}$.
Suppose \eqref{equ:SIndexCond} holds
for every nonempty subset $J\subset\left\{1,2,\ldots,m\right\}$.
Then the estimate \eqref{equ:TSigmBound} holds.
\end{thm}

\subsection{The second case: $0<p_i\le 1$ or $p_i=\infty$}\label{subsection:SecondCase}
\begin{thm}\label{thm:HpiLInf}
  Let $\frac{n}2<s_1,\ldots,s_m<\infty,$ $0<p_1,\ldots,p_l\le 1$,
$1\le l< m$,
and $\frac{1}{p_1}+ \cdots + \frac{1}{p_l} = \frac{1}{p}$.
Suppose \eqref{equ:SIndexCond} holds
for every nonempty subset $J\subset\left\{1,2,\ldots,l\right\}$.
Then
  \begin{equation}\label{equ:HpiLInf}
  \norm{T_{\sigma}}_{H^{p_1}\times\cdots\times H^{p_l}\times L^{\infty}\times\cdots\times L^{\infty}\longrightarrow  L^p}\lesssim \sup_{j\in\mathbb Z}{\left\Vert\sigma(2^j\cdot)\widehat\psi\, \right\Vert}_{W^{(s_1,\ldots,s_m)}}.
  \end{equation}
\end{thm}

\begin{proof}
The proof of this theorem is very similar to an analogous proof
of Theorem \ref{thm:MainSmallInd} that is given in \cite{GraHan},
but we retake it to provide certain detailed computations, adapted appropriately in our setting.

  By regularization (see \cite[Section 3]{GraHan}), we can always assume
that the inverse Fourier transform of $\sigma$ is smooth and
compactly supported.
The aim is to show that
  \begin{equation}\label{equ:TSigSmP}
    \norm{T_{\sigma}(f_1,\ldots,f_m)}_{L^p} \lesssim
A     \norm{f_1}_{H^{p_1}}\cdots\norm{f_l}_{H^{p_l}}
    \prod_{i=l+1}^m\norm{f_i}_{L^{\infty}}.
  \end{equation}
  Fix functions $f_i\in H^{p_i}.$
  Using atomic representations for $H^{p_i}$-functions, write
  $$
  f_i = \sum_{k_i\in\mathbb Z}\lambda_{i,k_i}a_{i,k_i},\quad 1\le i\le l,
  $$
  where $a_{i,k_i}$ are $L^\infty$-atoms for $H^{p_i}$ satisfying
  $$
  \supp(a_{i,k_i})\subset Q_{i,k_i},\quad \norm{a_{i,k_i}}_{L^{\infty}}\le {\left\vert{Q_{i,k_i}}\right\vert}^{-\frac{1}{p_i}},\quad \int_{Q_{i,k_i}}x^{\alpha}a_{i,k_i}(x)dx=0
  $$
  for ${\left\vert{\alpha}\right\vert}<N_i$ with $N_i$ large enough,
and $\sum_{k_i}{\left\vert{\lambda_{i,k_i}}\right\vert}^{p_i}
\le 2^{p_i}\norm{f_i}_{H^{p_i}}^{p_i}$.

  For a cube $Q $ we denote by $Q^*$ its dilation   by the factor $2\sqrt{n}.$
  Since $K=\sigma^{\vee}$ is smooth and compactly supported, Proposition \ref{pro:SwitchSum} yields that
  \begin{equation*}
  T_{\sigma}(f_1,\ldots,f_m)(x) = \sum_{k_1}\cdots\sum_{k_l}\lambda_{1,k_1}\ldots \lambda_{l,k_l} T_{\sigma}(a_{1,k_1},\ldots,a_{l,k_l},f_{l+1},\ldots,f_m)(x)
  \end{equation*}
  for a.e. $x\in\mathbb R^n.$  Now we can split $T_{\sigma}(f_1,\ldots,f_m)$
into two parts and estimate
  $$
  {\left\vert{T_{\sigma}(f_1,\ldots,f_m)(x)}\right\vert}\le G_1(x) +G_2(x),
  $$
  where
  $$
  G_1 = \sum_{k_1}\cdots\sum_{k_l}
{\left\vert{\lambda_{1,k_1}}\right\vert}\cdots
  {\left\vert{\lambda_{l,k_l}}\right\vert}
{\left\vert{T_{\sigma}(a_{1,k_1},\ldots,a_{l,k_l},f_{l+1},\ldots,f_m)}
\right\vert}\chi_{Q^*_{1,k_1}\cap\cdots\cap Q^*_{l,k_l}}
  $$
  and
  $$
  G_2 = \sum_{k_1}\cdots\sum_{k_l}{\left\vert{\lambda_{1,k_1}}\right\vert}\cdots
  {\left\vert{\lambda_{l,k_l}}\right\vert} {\left\vert{T_{\sigma}(a_{1,k_1},\ldots,a_{l,k_l},f_{l+1},\ldots,f_m)}
  \right\vert}\chi_{(Q^*_{1,k_1}\cap\cdots\cap Q^*_{l,k_l})^c} .
  $$

The first part $G_1(x)$ can be dealt via the argument in \cite{GrKal}
(reprised more clearly in \cite{GraHan}).
Suppose
the cubes $Q^*_{1,k_1}, \, \dots,\,Q^*_{l,k_l}$
satisfy $Q^*_{1,k_1}\cap\cdots\cap Q^*_{l,k_l} \neq \emptyset$.
From these cubes, choose a cube that has the minimum sidelength,
and denote it by
$R_{k_1,\ldots,k_l}$.
Then
$$
  Q^*_{1,k_1}\cap\cdots\cap Q^*_{l,k_l} \subset R_{k_1,\ldots,k_l}
  \subset Q^{\ast\ast}_{1,k_1}\cap\cdots\cap Q^{\ast\ast}_{l,k_l},
  $$
where $Q^{\ast\ast}_{i,k_i}$ denotes suitable dilation of $Q^{*}_{i,k_i}.$
We shall prove
\begin{equation}\label{equ:EstimateAvg}
\begin{split}
&   \frac{1}{\abs{R_{k_1,\ldots,k_l}}}
   \int_{R_{k_1,\ldots,k_l}}
   {\left\vert{T_{\sigma}(a_{1,k_1},\ldots,a_{l,k_l},f_{l+1},\ldots,f_m)(x)}\right\vert}\, dx
\\
&   \lesssim
   A\prod_{i=1}^l{\left\vert{Q_{i,k_i}}\right\vert}^{-\frac1{p_i}}
    \prod_{i=l+1}^m\norm{f_i}_{L^\infty}.
\end{split}
\end{equation}
To show this, assume without loss of generality
$R_{k_1,\ldots,k_l}=Q_{1,k_1}^{\ast}$.
Then, using  Theorem \ref{thm:L2estimate}, we have
  \begin{align*}
    &\int_{R_{k_1,\ldots,k_l}}
    {\left\vert{T_{\sigma}(a_{1,k_1},\ldots,a_{l,k_l},f_{l+1},\ldots,f_m)(x)}\right\vert}dx
    \\
    &\le \norm{T_{\sigma}(a_{1,k_1},\ldots,a_{l,k_l},f_{l+1},\ldots,f_m)
    }_{L^2}{\left\vert{R_{k_1,\ldots,k_l}}\right\vert}^{\frac12}\\
    &\lesssim A{\left\vert{R_{k_1,\ldots,k_l}}\right\vert}^{\frac12}\norm{a_{1,k_1}}_{L^2}
    \prod_{i=2}^l\norm{a_{i,k_i}}_{L^\infty}
    \prod_{i=l+1}^m\norm{f_i}_{L^\infty}
    \\
    &\le
    A{\left\vert{R_{k_1,\ldots,k_l}}\right\vert}^{\frac12}{\left\vert{Q_{1,k_1}}\right\vert}^{\frac12}
    \prod_{i=1}^l{\left\vert{Q_{i,k_i}}\right\vert}^{-\frac1{p_i}}
    \prod_{i=l+1}^m\norm{f_i}_{L^\infty}
    \\
    &\le
    A{\left\vert{R_{k_1,\ldots,k_l}}\right\vert}
    \prod_{i=1}^l{\left\vert{Q_{i,k_i}}\right\vert}^{-\frac1{p_i}}
    \prod_{i=l+1}^m\norm{f_i}_{L^\infty},
\end{align*}
which implies \eqref{equ:EstimateAvg}.
Now using Lemma \ref{lem:LpAvg}, the estimate \eqref{equ:EstimateAvg},
and H\"older's inequality, we obtain
\begin{align*}
    \norm{G_1}_{L^p}
   &\lesssim
   \bigg\Vert
   \sum_{k_1}\cdots\sum_{k_m}
   \bigg(\prod_{i=1}^{l} \abs{\lambda_{i,k_i}} \bigg)
    \frac{1}{\abs{R_{k_1,\ldots,k_l}}}
   \int_{R_{k_1,\ldots,k_l}}
   \\
   &\quad \quad \quad\quad
   \abs{T_{\sigma}  (a_{1,k_1},\ldots,a_{l,k_l},f_{l+1},\ldots,f_m)(x) }\, dx\,
   \chi_{R_{k_1,\ldots,k_l}}
   \bigg\Vert
   _{L^p}
   \\
   &\lesssim
   A\norm{
   \sum_{k_1}\cdots\sum_{k_m}
   \bigg(\prod_{i=1}^{l} \abs{\lambda_{i,k_i}}
   \abs{Q_{i,k_i}}^{-\frac1{p_i}}\chi_{Q^{\ast\ast}_{i,k_i}}\bigg)
    \prod_{i=l+1}^m\norm{f_i}_{L^\infty}
    }_{L^p}
\\
&=
A\norm{
\prod_{i=1}^{l}
\bigg(\sum_{k_i}
\abs{\lambda_{i,k_i}} \abs{Q_{i,k_i}}^{-\frac1{p_i}}
\chi_{Q^{\ast\ast}_{i,k_i}}\bigg)
}_{L^p}
\prod_{i=l+1}^m
\norm{f_i}_{L^\infty}
\\
&
\le
A\prod_{i=1}^{l}
\norm{
\sum_{k_i}
\abs{\lambda_{i,k_i}} \abs{Q_{i,k_i}}^{-\frac1{p_i}}
\chi_{Q^{\ast\ast}_{i,k_i}}
}_{L^{p_i}}
\prod_{i=l+1}^m
\norm{f_i}_{L^\infty}
\\
&\lesssim
A\prod_{i=1}^{l}
\norm{f_i}_{H^{p_i}}
\prod_{i=l+1}^m
\norm{f_i}_{L^\infty}.
\end{align*}
Thus we have
  \begin{equation}\label{equ:G1Part}
  \norm{G_1}_{L^{p}}\lesssim A\norm{f_1}_{H^{p_1}}\cdots\norm{f_m}_{H^{p_m}}.
  \end{equation}

  Now for the more difficult part, $G_2(x),$ we first restrict
  $x\in (\bigcap_{i\notin J}Q_{i,k_{i}}^*)\setminus(\bigcup_{i\in J}Q_{i,k_{i}}^*)$ for some nonempty subset $J\subset\left\{1,2,\ldots,l\right\}.$ To continue, we   need the following lemma whose proof is given in Section \ref{section:ProofLem}.
\begin{lem}\label{lem:KeyLem1}
    Let $\frac{n}2<s_1, \dots, s_m<\infty$,
$0<p_1, \dots, p_l\le 1$, $1\le i\le l<m$, and suppose
\eqref{equ:SIndexCond} holds
for all $J\subset\set{1,\ldots,l}$.
Let ${\sigma}$ be a function satisfying \eqref{equ:SigmCond}.
Suppose $a_i$ ($i=1,\ldots,l$) are atoms supported in the
cube $Q_i$ such that
   $$
   \norm{a_i}_{L^{\infty}}\le {\left\vert{Q_i}\right\vert}^{-\frac1{p_i}},\qquad \int_{Q_i}x^{\alpha}a_i(x)dx=0
   $$
   for all ${\left\vert{\alpha}\right\vert}< N_i$ with $N_i$
   sufficiently large.
   Fix a non-empty subset $J_0\subset\left\{1,\ldots,l\right\}.$
  Then there exist positive functions $b_1,\ldots,b_{l}$
such that $b_i$ depends only on
$m$, $n$, $(s_i)_{i=1, \dots, m}$, $(p_i)_{i=1, \dots, m}$,
$\sigma$, $J_0$, $N_i$, and $Q_i$,
and
    \begin{equation*}
    {\left\vert{T_{\sigma}(a_1,\ldots,a_l,f_{l+1},\ldots,f_m)(x)}
\right\vert}
\lesssim A\,
b_1(x)\cdots b_{l}(x)\norm{f_{l+1}}_{{L^\infty}}
\cdots\norm{f_{m}}_{{L^\infty}}
    \end{equation*}
    for all $x\in (\bigcap_{i\notin J_0}Q_i^*)
\setminus (\bigcup_{i\in J_0}Q_i^*),$
    and
    $\norm{b_i}_{L^{p_i}}\lesssim 1$,  $1\le i\le l .$
  \end{lem}
For each nonempty subset $J$ of $\{1,2,\dots,l\}$,  Lemma \ref{lem:KeyLem1}
guarantees the existence of positive functions $b_{1,k_1}^J,\ldots,b_{l,k_l}^J$
depending on $Q_{1,k_1},\ldots,Q_{l,k_l}$ respectively, such that
  \begin{equation*}
  {\left\vert{T_{\sigma}(a_{1,k_1},\ldots,a_{l,k_l},
  f_{l+1},\ldots,f_m)}\right\vert} \lesssim A\,  b_{1,k_1}^J \cdots b_{l,k_l}^J
  \prod_{i=l+1}^{\infty}\norm{f_i}_{L^{\infty}}
  \end{equation*}
  for all $x\in (\bigcap_{i\notin J}Q_{i,k_{i}}^*)\setminus(\bigcup_{i\in J}Q_{i,k_{i}}^*)$
  and $\big\|{b_{i,k_i}^J}\big\|_{L^{p_i}} \lesssim 1.$ Now set
  $$
  b_{i,k_i} = \sum_{\emptyset \ne J\subset\left\{1,2,\ldots,l\right\}}b_{i,k_i}^J.
  $$
  Then
  \begin{equation}\label{equ:BiComp}
  {\left\vert{T_{\sigma}(a_{1,k_1},\ldots,a_{l,k_l},
  f_{l+1},\ldots,f_m)}\right\vert}
  \chi_{(Q^*_{1,k_1}\cap\ldots\cap Q^*_{l,k_l})^c}
  \lesssim A\,  b_{1,k_1} \cdots b_{l,k_l}\prod_{i=l+1}^{\infty}\norm{f_i}_{L^{\infty}}
  \end{equation}
  and $\norm{b_{i,k_i}}_{L^{p_i}} \lesssim 1.$  Estimate   \eqref{equ:BiComp} yields
  $$
  G_2(x)\lesssim A\prod_{i=1}^l\left( \sum_{k_i}{\left\vert{\lambda_{i,k_i}}\right\vert}b_{i,k_i}(x) \right)\prod_{i=l+1}^{\infty}\norm{f_i}_{L^{\infty}}.
  $$
  Then apply H\"older's inequality to deduce  that
  \begin{equation}\label{equ:G2Part}
  \norm{G_2}_{L^p}\lesssim A\, \norm{f_1}_{H^{p_1}}\cdots\norm{f_l}_{H^{p_l}}
  \prod_{i=l+1}^{\infty}\norm{f_i}_{L^{\infty}}.
  \end{equation}
  Combining \eqref{equ:G1Part} and \eqref{equ:G2Part},
we obtain
\eqref{equ:TSigSmP} as needed.
This completes the proof.
\end{proof}

\subsection{The third case: $0<p_i\le 1$ or $2\le p_i\le \infty$}\label{subsection:ThirdCase}
\begin{thm}\label{thm:HpiLpiLp}
  Let $\frac{n}2<s_1,\ldots,s_m<\infty,$ $p_1, \ldots, p_m\in (0,1]\cup[2,\infty]$,
  $0<p<\infty$, and $\frac{1}{p_1}+ \cdots + \frac{1}{p_m} = \frac{1}{p}.$
  Assume there exists at least one index $i$ such that $p_i \in (0,1]$
  and also assume the condition \eqref{equ:SIndexCond} holds
  for every nonempty subset $J\subset\left\{1,2,\ldots,m\right\}.$
  Then the estimate \eqref{equ:TSigmBound} holds.
\end{thm}
\begin{proof}
In addition to the assumptions of the theorem,
we also assume
there exists at least one $i$ such that $p_i \in [2,\infty),$ since
otherwise the claim is already covered by Theorems \ref{thm:MainSmallInd} or \ref{thm:HpiLInf}.
Thus without loss of generality, we may assume that
$0<p_1,\ldots,p_l\le 1,$ $2\le p_{l+1},\ldots,p_\rho<\infty$, $p_{\rho+1}=\cdots=p_m=\infty$,
$1\le l < \rho \le m$, and
$\frac{1}{p_1}+ \cdots + \frac{1}{p_\rho} = \frac{1}{p}.$
Our goal is to establish the estimate
\begin{equation}\label{equ:HpiLpiLinf}
\norm{T_{\sigma}}_{H^{p_1}\times\cdots\times H^{p_l}
\times L^{p_{l+1}}\times\cdots\times L^{p_\rho}
\times L^\infty\times\cdots\times L^\infty
\longrightarrow  L^p}
\lesssim \sup_{j\in\mathbb Z}
{\left\Vert\sigma(2^j\cdot)\widehat\psi\, \right\Vert}_{W^{(s_1,\ldots,s_m)}}.
\end{equation}
Assume momentarily the validity of the following estimate
  \begin{equation}\label{equ:HpiL2Linf}
  \norm{T_{\sigma}}
  _{H^{p_1}\times\cdots\times H^{p_l}\times
  \underset{(\rho-l)-\mbox{\tiny times}}
  {\underbrace{L^{2}\times\cdots\times L^{2}}}\times
  \underset{(m-\rho)-\mbox{\tiny times}}
  {\underbrace{L^\infty\times\cdots\times L^\infty}}
  \longrightarrow  L^p}
  \lesssim \sup_{j\in\mathbb Z}
  {\left\Vert\sigma(2^j\cdot)\widehat\psi\, \right\Vert}
  _{W^{(s_1,\ldots,s_m)}}.
  \end{equation}
Then using Theorem \ref{thm:CalTorInterp} to
interpolate between \eqref{equ:HpiL2Linf} and \eqref{equ:HpiLInf},
we obtain the estimate \eqref{equ:HpiLpiLinf} as required.
(In fact, since the condition \eqref{equ:SIndexCond}
with $(p_i)_{i=1,\dots,m}$ in the estimates of \eqref{equ:HpiLInf},
\eqref{equ:HpiLpiLinf}, and \eqref{equ:HpiL2Linf}
give the same restriction on $(s_i)_{i=1,\dots, m}$,
in order to deduce \eqref{equ:HpiLpiLinf} from
\eqref{equ:HpiL2Linf} and \eqref{equ:HpiLInf},
we may fix $(s_i)_{i=1,\dots, m}$ and could
use the usual real or complex interpolation for linear operators.)
Thus it suffices to prove \eqref{equ:HpiL2Linf}.
In the rest of the proof, we assume $p_{l+1}=\cdots=p_{\rho}=2$.

Before we proceed to the proof of \eqref{equ:HpiL2Linf},
we shall see that it is sufficient to consider $\sigma$ that has support
in some cone.
To see this, for
$\eta =(\eta_1, \ldots, \eta_m)\in  {\mathbb R}^{mn}$,
consider the $m+1$ vectors
$
\eta_1,\, \ldots,\, \eta_m,\,
\eta_{m+1}=\sum_{i=1}^{m}\eta_i \in {\mathbb R}^{n}. $
If $\eta $ belongs to the unit sphere
$\Sigma = \{\eta \in \mathbb R ^{mn} \mid |\eta|=1\}$,
at least two of these $m+1$ vectors are not zero.
Hence, by the compactness of $\Sigma$, there exists a constant $a>0$
such that $\Sigma$ is covered by the $\binom{m+1}{2}$ open sets
$$
V(k_1, k_2) = \{\eta \in \Sigma :\,\, \abs{\eta_{k_1}}>a, \; \abs{\eta_{k_2}}>a\},
\quad 1\le k_1 <k_2\le m+1.
$$
We take a smooth partition of unity $\{\varphi_{k_1, k_2}\}$ on $\Sigma$ such that
$\mathrm{supp}\, \varphi_{k_1, k_2}\subset V(k_1, k_2)$ and decompose the multiplier
$\sigma$ as
$$
\sigma (\xi)
=\sum_{1\le k_1 <k_2\le m+1}
\sigma (\xi) \varphi_{k_1, k_2} (\xi / \abs{\xi})
=\sum_{1\le k_1 <k_2\le m+1}
\sigma _{k_1, k_2}(\xi).
$$
Then
$$
\mathrm{supp}\, \sigma_{k_1, k_2}
\subset \Gamma (V(k_1, k_2))
= \{\xi \in \mathbb R ^{mn} \setminus \{0\}
:\,\, \xi /\abs{\xi} \in V(k_1, k_2)\}
$$
and Lemma \ref{lem:ProdSobN} gives
$$
  \sup_{j\in\mathbb Z}
  \left\Vert{\sigma _{k_1, k_2}
  (2^{j}\cdot) \widehat{\psi}}\,
  \right\Vert_{W^{(s_1,\ldots,s_m)}}
  \lesssim
  \sup_{j\in\mathbb Z}
  \left\Vert{\sigma(2^{j}\cdot)\widehat\psi}\,
  \right\Vert_{W^{(s_1,\ldots,s_m)}}.
$$
The estimate \eqref{equ:TSigmBound} follows if we prove it with
$\sigma_{k_1,k_2}$ in place of $\sigma$.
This means that it is sufficient to prove \eqref{equ:TSigmBound}
under the additional assumption that
\begin{equation}\label{equ:support}
\mathrm{supp}\, \sigma \subset \Gamma (V(k_1, k_2))
\end{equation}
for some $1\le k_1 <k_2\le m+1$.

To simplify notation, we also assume
\begin{equation}\label{equ:normalize}
\sup_{j\in\mathbb Z}
  \left\Vert{\sigma (2^{j}\cdot) \widehat{\psi}}\,
  \right\Vert_{W^{(s_1,\ldots,s_m)}}
  = 1
\end{equation}
and  write
\begin{equation}\label{equ:SigmaDecomposition}
\sigma= \sum_{j\in \mathbb Z} \sigma_{j},
\quad
\sigma_{j} (\xi) = \sigma (\xi) \widehat{\psi} (2^{-j} \xi).
\end{equation}

We shall divide the proof into two cases.
First case: $\sigma$ satisfies \eqref{equ:support} with
$1\le k_1 < k_2 \le m$.
Second case: $\sigma$ satisfies \eqref{equ:support} with
$1\le k_1 \le m$ and $k_2=m+1$.
In the first case, we shall directly prove the estimate
\begin{equation}\label{equ:FirstGoal}
\norm{
T_{\sigma} (f_1, \ldots, f_m)
}_{L^p}
\lesssim \prod_{i=1}^{m}
\left\Vert
f_i
\right\Vert_{H^{p_i}}.
\end{equation}
In the second case, we shall use a Littlewood-Paley function.
Notice that, in the second case, the support of the Fourier transform of
$T_{\sigma_{j}} (f_1, \ldots, f_m) $ is included in the annulus
$\{\xi \in \mathbb R ^{n} :\,\, B^{-1} 2^{j} \le \abs{\xi} \le B 2^{j}\}$
with some constant $B>1$.
Hence, by Lemma \ref{lem:LittlewoodPaley},
we have
\begin{equation}\label{equ:LittlewoodPaley}
\left\Vert
T_{\sigma}(f_1, \ldots, f_m)
\right\Vert_{H^p}
\lesssim
\left\Vert
\bigg(
\sum_{j\in \mathbb{Z}}
\abs{T_{\sigma_j}(f_1, \ldots, f_m)}
^{2}
\bigg)^{1/2}
\right\Vert_{L^p}.
\end{equation}
Thus, in the second case,
we shall consider the function
$$
GT_{\sigma}(f_1, \ldots, f_m)
=\bigg(
\sum_{j\in \mathbb{Z}}
\abs{T_{\sigma_j}(f_1, \ldots, f_m)}
^{2}
\bigg)^{1/2}
$$
and prove the estimate
\begin{equation}\label{equ:SecondGoal}
\norm{
GT_{\sigma} (f_1, \ldots, f_m)
}_{L^p}
\lesssim \prod_{i=1}^{m}
\left\Vert
f_i
\right\Vert_{H^{p_i}},
\end{equation}
which combined with \eqref{equ:LittlewoodPaley} implies
\eqref{equ:FirstGoal}.

The essential part of the proofs of
\eqref{equ:FirstGoal} and \eqref{equ:SecondGoal} are
given in the following two lemmas.

\begin{lem}\label{lem:KeyLem2}
  Let $\frac{n}2<s_1,\ldots,s_m<\infty,$
  $0<p_1, \ldots, p_l \le 1$,
  $p_{l+1}, \ldots, p_{\rho} =2$,
  $p_{\rho+1}, \ldots, p_{m} =\infty$,
  $1\le l < \rho \le m$,
  and suppose \eqref{equ:SIndexCond}
  holds for every nonempty subset $J\subset \{1,\ldots, l\}$.
  Let $a_i$, $1\le i \le l$, be $H^{p_i}$ atoms such that
  $$
  \mathrm{supp}\, a_i \subset Q_i,
  \quad
  \norm{a_i}_{L^\infty}
  \le \abs{Q_i}^{-1/p_i},
  \quad
  \int a_i (x) x^{\alpha}\, dx =0
  $$
  for $\abs{\alpha}< N_i$, where $N_i$ is a sufficiently
  large positive integer
  and $Q_i$ is a cube.
  Let $f_{l+1}, \ldots, f_{\rho}\in L^2$
  and $f_{\rho+1}, \ldots, f_{m}\in L^\infty$.
  Finally suppose $\sigma$ satisfies \eqref{equ:normalize}
  and \eqref{equ:support} with some
  $1\le k_1<k_2\le m$.
  Then there exist functions
  $b_1, \ldots, b_l$ and $\widetilde{f}_{l+1}, \ldots, \widetilde{f}_{\rho}$
  such that
  \begin{equation}\label{equ:pointwiseFirstCase}
  \abs{T_{\sigma}(a_1,\ldots, a_l, f_{l+1}, \ldots, f_{m})(x)}
  \lesssim
  \prod_{i=1}^{l} b_i(x)\cdot
  \prod_{i=l+1}^{\rho} \widetilde{f}_i(x)
  \cdot
  \prod_{i=\rho+1}^{m} \norm{f_i}_{L^\infty};
  \end{equation}
  the function $b_i$ depends only on
  $m$, $n$, $(s_i)_{i=1,\dots, m}$, $(p_i)_{i=1,\dots, m}$,
  $\sigma$, $i$, $a_i$,
  and $(f_i)_{i=\rho+1, \dots,m}$;
  the function $\widetilde{f}_i$
  depends only on
  $m$, $n$, $(s_i)_{i=1,\dots, m}$,
  $i$, $f_i$,
  and $(f_i)_{i=\rho+1, \dots,m}$;
  and they satisfy the estimates
  $\norm{b_i}_{L^{p_i}}\lesssim 1$
  and $\| \widetilde{f}_i  \|_{L^{2}}\lesssim \norm{f_i}_{L^2}$.
\end{lem}
\begin{lem}\label{lem:KeyLem3}
  Let $s_i$, $p_i$, $a_i$, and $f_i$ be the same as
in Lemma \ref{lem:KeyLem2}.
  Suppose $\sigma$ satisfies \eqref{equ:normalize} and
  \eqref{equ:support} with some $1\le k_1\le m$ and $k_2=m+1$.
  Then there exist functions
  $b_1, \ldots, b_l$ and $\widetilde{f}_{l+1}, \ldots, \widetilde{f}_{\rho}$
  that satisfy
  $$
  GT_{\sigma}(a_1,\ldots, a_l, f_{l+1}, \ldots, f_{m})(x)
  \lesssim
  \prod_{i=1}^{l} b_i(x)\cdot
  \prod_{i=l+1}^{\rho} \widetilde{f}_i(x)
  \cdot
  \prod_{i=\rho+1}^{m} \norm{f_i}_{L^\infty}
  $$
  and have the same properties as
  in Lemma \ref{lem:KeyLem2}.
  \end{lem}
The proofs of these lemmas will be given in
Section \ref{section:ProofLem}.
We shall continue the proof of Theorem \ref{thm:HpiLpiLp}.
To utilize the above lemmas, we decompose
$f_i\in H^{p_i}$, $1\le i\le l$, into atoms as
$f_i = \sum_{k_i\in\mathbb Z}\lambda_{i,k_i}
  a_{i,k_i}$
  with $\lambda_{i,k_i}$, $a_{i,k_i}$, and the cubes
  $Q_{i,k_i}$ being the same as in the proof of
  Theorem \ref{thm:HpiLInf}.

Consider the first case where
$\sigma$ satisfies \eqref{equ:support} with
$1\le k_1 < k_2 \le m$.
In this case, Lemma \ref{lem:KeyLem2} yields functions
$b_{i,k_i}$ ($1\le i\le l$, $k_i\in \mathbb Z$)
and
$\widetilde{f}_i$ ($l +1 \le i \le \rho$) such that
  $$
  \abs{T_{\sigma}(a_{1,k_1},\ldots, a_{l,k_l},
f_{l+1}, \ldots, f_{m})(x)}
  \lesssim
  \prod_{i=1}^{l} b_{i,k_i}(x)\cdot
  \prod_{i=l+1}^{\rho} \widetilde{f}_i(x)
  \cdot
  \prod_{i=\rho+1}^{m} \norm{f_i}_{L^\infty}
  $$
and
$\norm{b_{i,k_i}}_{L^{p_i}}\lesssim 1$
and
$ \|{\widetilde{f}_i}\|_{L^{2}}\lesssim \norm{f_i}_{L^2}$.
Notice that $b_{i,k_i}$ do not depend on
$k_j$ with $j \neq i$ and
$\widetilde{f}_i$ do not depend on $k_1, \ldots, k_l$.
Hence, by the multilinear property of the operator
$T_{\sigma}$, we have
\begin{align*}
&
\abs{
T_{\sigma}(f_1, \ldots, f_m)(x)
}
\\
&\lesssim
\sum_{k_1}\cdots\sum_{k_l}
\abs{
\lambda_{1,k_1}\ldots \lambda_{l,k_l}
}
\prod_{i=1}^{l} b_{i,k_i}(x)\cdot
\prod_{i=l+1}^{\rho} \widetilde{f}_i(x)
\cdot
\prod_{i=\rho+1}^{m} \norm{f_i}_{L^\infty}
\\
&=
\prod_{i=1}^{l}
\bigg(
\sum_{k_i}
\abs{
\lambda_{i,k_i}}
b_{i,k_i}(x)
\bigg)
\cdot
\prod_{i=l+1}^{\rho} \widetilde{f}_i(x)
\cdot
\prod_{i=\rho+1}^{m} \norm{f_i}_{L^\infty}.
\end{align*}
(We omit necessary a limiting argument to the treat infinite sum,
which could be achieved with the aid of
Proposition \ref{pro:SwitchSum}.)
For $1\le i\le l$, we have
\begin{equation*}
\norm{
\sum_{k_i}
\left\vert
\lambda_{i,k_i}\right\vert
b_{i,k_i}
}_{L^{p_i}}^{p_i}
\le
\sum_{k_i}
\abs{\lambda_{i,k_i}}^{p_i}
\norm{b_{i,k_i}}_{L^{p_i}}^{p_i}
\lesssim
\sum_{k_i}
\abs{\lambda_{i,k_i}}^{p_i}
\lesssim
\norm{f_i}_{H^{p_i}}^{p_i}.
\end{equation*}
The above pointwise inequality and
H\"older's inequality
now give \eqref{equ:FirstGoal}.

Next consider the second case
where $\sigma$ satisfies \eqref{equ:support} with
$1\le k_1\le m$ and $k_2=m+1$.
By the sublinear property of square function,
we have
\begin{equation*}
GT_{\sigma}(f_1, \ldots, f_m)(x)
\le
\sum_{k_1}\cdots\sum_{k_l}
\abs{
\lambda_{1,k_1}\cdots \lambda_{l,k_l}
}
GT_{\sigma}(a_{1,k_1},\ldots, a_{l,k_l},
f_{l+1}, \ldots, f_{m})(x).
\end{equation*}
(Again we omit the necessary limiting argument.)
Hence, using Lemma \ref{lem:KeyLem3} and arguing in the same way
as in the first case,
we obtain \eqref{equ:SecondGoal}.
Thus the proof of Theorem \ref{thm:HpiLpiLp} is reduced to
Lemmas \ref{lem:KeyLem2} and \ref{lem:KeyLem3}.
\end{proof}

\subsection{The last case: $0<p_i \le \infty$.}\label{subsection:LastCase}
In this subsection, we shall prove the estimate \eqref{equ:TSigmBound}
for the entire range $0<p_i \le \infty$.
Since the necessity of the conditions
$s_i \ge n/2$ and \eqref{equ:SIndexCondGe} has already been shown
in \cite[Theorem 5.1]{GraHan},
this will complete the proof of Theorem \ref{thm:General}.
To simplify notation, we use the letters
${\bf s}$ and ${\bf p}$ to denote
$(s_1, \dots, s_m)$
and $(p_1, \dots, p_m)$, respectively.

We shall slightly change the formulation
of the claim of Theorem \ref{thm:General}.
We assume
$0<p_1, \dots , p_m\le \infty$,
\begin{equation}\label{equ:halfnGe}
\infty>s_1, \dots, s_m \ge n/2,
\end{equation}
and assume they satisfy
\eqref{equ:SIndexCondGe} for all
nonempty subset $J\subset \{1, \dots, m\}$.
We shall prove the estimate
\begin{equation}\label{equ:TSigmBoundEpsilon}
\norm{
T_{\sigma}
}
_{H^{p_1}\times \cdots \times H^{p_m}\to L^p}
\lesssim
\sup_{j\in {\mathbb Z}}
\norm{\sigma (2^{j} \cdot )\widehat{\psi}}\,
_{W^{(s_1+\epsilon, \cdots, s_m + \epsilon)}}
\end{equation}
holds for every $\epsilon>0$,
where $1/p=1/p_1+ \dots + 1/p_m$
and the space $L^p$ should be replaced by $BMO$
if $p_1=\dots =p_m=p=\infty$.
This is equivalent to the estimate
given in Theorem \ref{thm:General}.
The proof will be given in two steps.

In the first step, we fix $\bf s$ satisfying
\eqref{equ:halfnGe} and consider the set
$\Delta ({\bf s})$ that consists of all
$(1/p_1, \dots, 1/p_m)\in [0,\infty)^m$ such that
the condition \eqref{equ:SIndexCondGe} holds for all
nonempty subset $J\subset \{1, \dots, m\}$.
We prove the following lemma.

\begin{lem}\label{lem:DeltaConvex}
If ${\bf s}$ satisfies \eqref{equ:halfnGe}, then
$\Delta ({\bf s})$ is the convex hull of the
point $(0,\dots, 0)$ and the points
$(1/p_1, \dots, 1/p_m)$ that satisfy
\begin{equation}\label{equ:extremeDelta1}
\text{ $1/p_i=0$ or $1/p_i=s_i/n$ or $1/p_i=s_i/n + 1/2$
for all $i$,}
\end{equation}
and
\begin{equation}\label{equ:extremeDelta2}
\text{ $1/p_i=s_i/n + 1/2$ for exactly one $i$.}
\end{equation}
\end{lem}

\begin{proof}Fix $\mathbf{s} = (s_1,\ldots,s_m)$ such that $s_i\ge \frac{n}2$ for all $1\le i\le m.$
Condition \eqref{equ:SIndexCondGe} gives a clearer presentation of the set $\Delta(\mathbf{s})$ as
$$
\Delta(m,\mathbf{s}) = \set{\Big(\frac1{p_1},\ldots,\frac1{p_m}\Big)\in\mathbb R^m :\ 0\le \frac1{p_i}\le \frac{s_i}n+\frac12, \ \ \sum_{i\in J}\frac1{p_i}\le \sum_{i\in J}\frac{s_i}{n}+\frac12},
$$
where $J$ runs over all non-empty subsets of $\set{1,\ldots,m}$. We let $H$ denote
 the convex hull of $(0,\dots , 0)$ and of all the points $(1/p_1,\dots , 1/p_m)$
 that satisfy \eqref{equ:extremeDelta1} and \eqref{equ:extremeDelta2}.
We will show that $\Delta(m,\mathbf{s})=H$ by induction in $m$.

The case when $m=2$ is trivial because $\Delta(2,\mathbf{s})$ is the convex hull of the following points $(0,0)$, $(\frac{s_1}n+\frac12,0)$, $(\frac{s_1}n+\frac12,\frac{s_2}n)$, $(0,\frac{s_2}n+\frac12)$ and $(\frac{s_1}n, \frac{s_2}n+\frac12)$; hence, the statement of Lemma \ref{lem:DeltaConvex} holds obviously in this case.

Now fix an $m>2$ and suppose that the statement of the lemma is true for $m-1.$ For $1\le k\le m$, denote
$$
\Delta^k(m,\mathbf{s}) = \set{\Big(\frac1{p_1},\ldots,\frac1{p_m}\Big)\in \Delta(m,\mathbf{s})  :\ 0\le \frac1{p_k}\le \frac{s_k}n},
$$
$$
F^k_0(m,\mathbf{s}) = \set{\Big(\frac1{p_1},\ldots,\frac1{p_m}\Big)\in \Delta(m,\mathbf{s}) :\ \frac1{p_k}=0},
$$
$$
F^k_1(m,\mathbf{s}) = \set{\Big(\frac1{p_1},\ldots,\frac1{p_m}\Big)\in \Delta(m,\mathbf{s}) :\ \frac1{p_k}=\frac{s_k}n},
$$
and
$$
\Delta^0(m,\mathbf{s}) = \set{\Big(\frac1{p_1},\ldots,\frac1{p_m}\Big)\in \Delta(m,\mathbf{s})  :\ \frac{s_i}n\le \frac1{p_i}\le \frac{s_i}n+\frac12,\ \forall\; 1\le i\le m}.
$$
It is easy to see that $\Delta(m,\mathbf{s}) = \cup_{k=0}^m\Delta^k(m,\mathbf{s}).$
We observe that $H$ is a subset of $\Delta(m,\mathbf{s})$, since each vertex of $H$ obviously sits inside the convex set $\Delta(m,\mathbf{s}).$ Thus, it suffices to prove that $\Delta^k(m,\mathbf{s})$ is a subset of $H$ for every $0\le k\le m.$

We first consider $\Delta^k(m,\mathbf{s})$ for $1\le k\le m.$
By induction, the face $F^k_0(m,\mathbf{s})$ is the convex hull of the following points $(0,\ldots,0)$ and $(\frac1{p_1},\ldots,\frac1{p_m})$, where $\frac1{p_k}=0,$ $\frac1{p_i}\in \set{0,\frac{s_i}n,\frac{s_i}n+\frac12}$ for $i\ne k,$ and there exists exactly one $i\ne k$ such that $\frac1{p_i} = \frac{s_i}n+\frac12.$
Similarly, the face $F^k_1(m,\mathbf{s})$ is determined by the same constraints for all variables $\frac1{p_i},i\ne k$ as those for $F^k_0(m,\mathbf{s}).$
Therefore, by induction, we have that $F^k_1(m,\mathbf{s})$ is the convex hull of the points $(0,\ldots,0,\frac{s_k}n,0,\ldots,0)$ and $(\frac1{p_1},\ldots,\frac1{p_m})$, where $\frac1{p_k}=\frac{s_k}n,$ $\frac1{p_i}\in \set{0,\frac{s_i}n,\frac{s_i}n+\frac12}$ for $i\ne k,$ and there exists exactly one $i\ne k$ such that $\frac1{p_i} = \frac{s_i}n+\frac12.$
Note that the point $(0,\ldots,0,\frac{s_k}n,0,\ldots,0)$ belongs to the line segment that joins the origin $(0,\dots , 0)$ with  $(0,\ldots,0,\frac{s_k}n+\frac12,0,\ldots,0).$ Thus $F^k_0(m,\mathbf{s})$ and $F^k_1(m,\mathbf{s})$ are contained in $H$, and hence, $\Delta^k(m,\mathbf{s})$ is a subset of $H$ since $\Delta^k(m,\mathbf{s})$ is a convex hull of two faces $F^k_0(m,\mathbf{s})$ and $F^k_1(m,\mathbf{s})$.

It remains to check that $\Delta^0(m,\mathbf{s})\subset H.$ In this case, we note that the constraints $0\le \frac1{p_i}-\frac{s_i}n\le \frac12,\ \forall\; 1\le i\le m$ and
$$
\sum_{i=1}^m\left( \frac1{p_i}-\frac{s_i}n\right)\le \frac12
$$
imply that $\Delta^0(m,\mathbf{s})$ is a standard $m$-simplex with vertices $(\frac{s_1}n,\ldots,\frac{s_m}n)$ and $(\frac1{p_1},\ldots,\frac1{p_m})$, where $\frac1{p_i}\in \set{\frac{s_i}n,\frac{s_i}n+\frac12}$ for $1\le i\le m,$ and there exists exactly one $i$ such that $\frac1{p_i} = \frac{s_i}n+\frac12,$ which implies $\Delta^0(m,\mathbf{s})\subset H$ with noting that the point $(\frac{s_1}n,\ldots,\frac{s_m}n)\in F^k_1(m,\mathbf{s})\subset H.$
\end{proof}

By virtue of
Lemma \ref{lem:DeltaConvex} and
Theorem \ref{thm:CalTorInterp},
to prove the estimate \eqref{equ:TSigmBoundEpsilon}
under the assumptions
\eqref{equ:halfnGe} and \eqref{equ:SIndexCondGe},
it is sufficient to show it
for ${\bf p}=(\infty, \dots, \infty)$ and
for
$\bf p$ satisfying \eqref{equ:extremeDelta1}
and \eqref{equ:extremeDelta2}.
For ${\bf p}=(\infty, \dots, \infty)$,
the estimate \eqref{equ:TSigmBoundEpsilon} with $BMO$ in place of $L^p$
is established in \cite[Corollary 6.3]{GraHan}.
Thus it is sufficient to
consider the latter points.
In the second step, we shall prove the following lemma,
which will complete the proof of Theorem \ref{thm:General}.

\begin{lem}\label{lem:Interpolation2}
Estimate \eqref{equ:TSigmBoundEpsilon} holds if
${\bf s}$ and ${\bf p}$
satisfy \eqref{equ:halfnGe}, \eqref{equ:extremeDelta1}, and
\eqref{equ:extremeDelta2}.
\end{lem}

\begin{proof}
For ${\bf p}\in (0,\infty]^{m}$,
we define $\ell ({\bf p})$ to be the
number of the indices $i \in \{1, \dots, m\}$ such that
$1<p_i<2$.
We shall prove the claim
by induction on $\ell ({\bf p})$.

The conditions \eqref{equ:halfnGe} and \eqref{equ:extremeDelta2}
imply in particular that there exists at least one $i$ such that $p_i \le 1$.
Hence if $\ell ({\bf p})=0$
then the claim directly follows from
Theorem \ref{thm:HpiLpiLp}.

Assume $\ell_0 \ge 1$ and
assume the claim holds if $\ell ({\bf p}) < \ell_0$.
Let
$$
({\bf p}^{0}, {\bf s}^{0})=
(p^{0}_{1}, \dots, p^{0}_{m}, s^{0}_{1}, \dots, s^{0}_{m})
$$
be a point that satisfies
the conditions
\eqref{equ:halfnGe}, \eqref{equ:extremeDelta1}, and
\eqref{equ:extremeDelta2},
and satisfies $\ell ({\bf p}^{0}) = \ell_0$.
There exists an index $i$ such that $1<p^{0}_{i}<2$.
Notice that $1/p^{0}_{i}=s^{0}_{i} / n$ for this index $i$.
Without loss of generality, we assume $1>1/p^{0}_{1}=s^{0}_{1} / n> 1/2$.
Then the condition \eqref{equ:extremeDelta2} implies that
there exists exactly one $i$ such that $2\le i\le m$ and
$1/p^{0}_{i}=s^{0}_{i} / n + 1/2$.
Consider the following two points:
\begin{align*}
({\bf p}^{\prime}, {\bf s}^{\prime})
&=(1, p^{0}_{2}, \dots, p^{0}_{m}, n, s^{0}_{2}, \dots, s^{0}_{m}),
\\
({\bf p}^{\prime\prime}, {\bf s}^{\prime\prime})
&=(2, p^{0}_{2}, \dots, p^{0}_{m}, n/2, s^{0}_{2}, \dots, s^{0}_{m}).
\end{align*}
Both $({\bf p}^{\prime}, {\bf s}^{\prime})$
and
$({\bf p}^{\prime\prime}, {\bf s}^{\prime\prime})$
satisfy
the conditions
\eqref{equ:halfnGe}, \eqref{equ:extremeDelta1}, and
\eqref{equ:extremeDelta2},
and $\ell ({\bf p}^{\prime}) =\ell ({\bf p}^{\prime\prime}) = \ell_0 -1$.
Hence by the induction hypothesis
the estimate \eqref{equ:TSigmBoundEpsilon}
holds for
$({\bf p}^{\prime}, {\bf s}^{\prime})$
and
$({\bf p}^{\prime\prime}, {\bf s}^{\prime\prime})$.
Then,
by Theorem \ref{thm:CalTorInterp},
it follows that
the estimate \eqref{equ:TSigmBoundEpsilon}
also holds for
$({\bf p}^{0}, {\bf s}^{0})$.
This completes the proof of Lemma \ref{lem:Interpolation2}.
\end{proof}

\section{Proofs of the key lemmas}\label{section:ProofLem}

  \begin{proof}[Proof of Lemma \ref{lem:KeyLem1}]
   Without loss of generality, we assume that $J_0=\left\{1,\ldots,r\right\}$ for some
   $1\le r\le l,$ and $\norm{f_i}_{L^{\infty}}=1$ for all $l+1\le i\le m.$ Fix
   $$x\in \Big(\bigcap_{i=r+1}^lQ_i^*\Big)\setminus\bigcup_{i=1}^rQ_i^*$$
   (when $r=l,$ just fix $x\in \mathbb R^n\setminus\bigcup_{i=1}^lQ_i^*$).
   Now we can write
    $$
    T_{\sigma}(a_1,\ldots,a_l, f_{l+1},\ldots,f_m)(x) = \sum_{j\in\mathbb Z}g_j(x),
    $$
    where $g_j(x)$ is the function
    $$
  \int_{\mathbb R^{mn}}2^{jmn}
  K_j(2^j(x-y_1), \ldots, 2^j(x-y_m))
  a_1(y_1)\cdots a_l(y_l) f_{l+1}(y_{l+1})\cdots f_m(y_m)\, d\vec{y}
    $$
    with $K_j = \big(\sigma(2^j\cdot)\widehat\psi\, \big)^{\vee}.$
    Let $c_i$ be the center of the cube $Q_i$ $(1\le i\le l).$
    For $1\le i\le r$, since $x\notin Q_i^*$,
    we have ${\left\vert{x-c_i}\right\vert}\approx {\left\vert{x-y_i}\right\vert}$
    for all $y_i\in Q_i$.
    Fix $1\le k\le r.$
    Using Lemma \ref{lem:LInfL2}
    and applying the Cauchy-Schwarz inequality we obtain
    \begin{align}
      \prod_{i=1}^r&\left<2^j(x-c_i)\right>^{s_i}
  {\left\vert{g_j(x)}\right\vert}\notag\\
      \lesssim
&\ 2^{jmn}\!\! \! \int\limits_{Q_1\times\cdots\times{Q_{l}}
\times {\mathbb R}^{(m-l)n}}
      \prod_{i=1}^r\left<2^j(x-y_i)\right>^{s_i}
{\left\vert{K_j(2^j(x-y_1),\ldots,2^j(x-y_m))}\right\vert}
 \prod_{i=1}^l\norm{a_i}_{L^{\infty}}\,d\vec y
\notag
\\
      \le
      &
      \ 2^{jmn}\prod_{i=1}^l{\left\vert{Q_i}\right\vert}^{-\frac1{p_i}}
      \!\! \! \int\limits_{
Q_1\times\cdots\times{Q_{r}}
\times {\mathbb R}^{(m-r)n}
}
      \prod_{i=1}^r\left<2^j(x-y_i)\right>^{s_i}
{\left\vert{K_j(2^j(x-y_1),\ldots,2^j(x-y_m))}\right\vert}
\, d\vec y
\notag
\\
= &
\ 2^{jrn}\prod_{i=1}^l{\left\vert{Q_i}\right\vert}^{-\frac1{p_i}}
      \int\limits_{Q_1\times\cdots\times{Q_r}\times \mathbb R^{(m-r)n}}
      \prod_{i=1}^r\left<2^j(x-y_i)\right>^{s_i}
      \notag\\
      &\times{\left\vert
{K_j(2^j(x-y_1),\ldots,2^j(x-y_r),y_{r+1},\ldots,y_m)}
\right\vert}
\,dy_1\cdots dy_rdy_{r+1}\cdots dy_m
\notag
\\
      \le &\
2^{jrn}\prod_{i=1}^r{\left\vert{Q_i}\right\vert}^{1-\frac1{p_i}}
      \prod_{i=r+1}^l{\left\vert{Q_i}\right\vert}^{-\frac1{p_i}}
\int_{\mathbb R^{(m-r)n}}
      \int\limits_{Q_k} {\left\vert{Q_k}\right\vert}^{-1}
\left<2^j(x-y_k)\right>^{s_k}\times
\notag
\\
      &\times\norm{\prod_{\substack{i=1\\i\ne k}}^{r}
\left<y_i\right>^{s_i}
K_j(y_1,\ldots,y_{k-1},2^j(x-y_k),y_{k+1}, \ldots,y_m)}
_{L^{\infty}(dy_1 \cdots \widehat{dy_k} \cdots dy_{r})}
   \hspace{-.8in}   dy_k dy_{r+1}\cdots dy_m
\notag\\
      \lesssim &\ 2^{jrn}\prod_{i=1}^r{\left\vert{Q_i}\right\vert}
^{1-\frac1{p_i}}
      \prod_{i=r+1}^l{\left\vert{Q_i}\right\vert}^{-\frac1{p_i}}
\int_{\mathbb R^{(m-r)n}}
      \int\limits_{Q_k}
      {\left\vert{Q_k}\right\vert}^{-1}\left<2^j(x-y_k)\right>^{s_k}
\times
\notag\\
      &\times\norm{\prod_{\substack{i=1\\i\ne k}}^{r}\left<y_i\right>^{s_i}
      K_j(y_1,\ldots,y_{k-1},2^j(x-y_k),y_{k+1}, \ldots,y_m)}
      _{L^{2}(dy_1 \cdots \widehat{dy_k} \cdots dy_{r})}
      \hspace{-.8in} dy_kdy_{r+1}\cdots dy_m\notag\\
      \lesssim &\ 2^{jrn}\prod_{i=1}^r{\left\vert{Q_i}\right\vert}^{1-\frac1{p_i}}
      \prod_{i=r+1}^l{\left\vert{Q_i}\right\vert}^{-\frac1{p_i}}
      \int\limits_{Q_k}
      {\left\vert{Q_k}\right\vert}^{-1}\left<2^j(x-y_k)\right>^{s_k}\times\notag\\
      &\times
      \norm{
      \prod_{\substack{i=1\\i\ne k}}^{m}
      \left<y_i\right>^{s_i}
      K_j(y_1,\ldots,y_{k-1},2^j(x-y_k),y_{k+1}, \ldots,y_m)
      }_{L^{2}(dy_1 \cdots \widehat{dy_k} \cdots dy_m)}
      \, dy_k
      \notag
      \\
      \label{ali:H10Func}
=
      &
      \ 2^{jrn}
      \bigg(\prod_{i=1}^r{\left\vert{Q_i}\right\vert}^{1-\frac 1{p_i}}\bigg)
      \bigg(\prod_{i=r+1}^lb_i(x)\bigg)  h_j^{(k,0)}(x),
    \end{align}
    where
    \begin{align*}
    h_j^{(k,0)}(x) = &\, \dfrac1{{\left\vert{Q_k}\right\vert}}
    \int\limits_{Q_k}\left<2^j(x-y_k)\right>^{s_k}\\
      & \hspace{-.5in}\times
      \norm{\prod_{\substack{i=1\\i\ne k}}^{m}\left<y_i\right>^{s_i}
      K_j(y_1,\ldots,y_{k-1},2^j(x-y_k),y_{k+1}, \ldots,y_m)}_{L^{2}(
      dy_1 \cdots \widehat{dy_k} \cdots dy_m)}dy_k
    \end{align*}
    and $b_i(x) = {\left\vert{Q_i}\right\vert}^{-\frac1{p_i}}\chi_{Q_i^*}(x)$
    for $r+1\le i\le l.$
The functions $b_i$, $r+1\le i\le l$, obviously satisfy the estimate
$\norm{b_i}_{L^{p_i}}\lesssim 1$.
    Minkowski's inequality gives
    $$
    \norm{h_j^{(k,0)}}_{L^2}\le 2^{-\frac{jn}{2}}
    {\left\Vert\sigma(2^j\cdot)\widehat\psi\right\Vert}
    _{W^{(s_1,\ldots,s_m)}} \le  A 2^{-\frac{jn}{2}}.
    $$
    Using the vanishing moment condition of $a_k$
    and Taylor's formula, we write
    $$
    \begin{aligned}
    g_j(x) \, = &\, 2^{jmn}\sum_{{\left\vert{\alpha}\right\vert}=N_k}
    C_{\alpha}\int_{\mathbb R^{mn}}
    \Bigg\{\int_0^1(1-t)^{N_k-1}\\
    &\times \partial^{\alpha}_{k}
    K_j\Big(2^j(x-y_1),\ldots,2^j x_{c_k,y_k}^t,\ldots, 2^j(x-y_m)\Big)\\
    &\times(2^j(y_k-c_k))^{\alpha}a_1(y_1)\cdots
    a_l(y_l)f_{l+1}(y_{l+1})\cdots f_m(y_m)\ dt\Bigg\}\ dy_1\cdots dy_m,
    \end{aligned}
    $$
    where
    $x_{c_k,y_k}^t=x-c_k-t(y_k-c_k)$ and $\partial^{\alpha}_{k}
    K_j (z_1, \dots, z_m)= \partial_{z_k}^{\alpha} K_j (z_1, \dots, z_m)$.
    Notice that $\abs{x_{c_k,y_k}^t}\approx \abs{x-c_k}$ for
    $x\not\in Q_k^{\ast}$, $y_k\in Q_k$, and $0<t<1$.
Repeating the preceding argument, we obtain
    \begin{equation}\label{equ:H1NFunc}
    \prod_{i=1}^r\left<2^j(x-c_i)\right>^{s_i}
    {\left\vert{g_j(x)}\right\vert}
    \lesssim 2^{jrn}\bigg(\prod_{i=1}^r{\left\vert{Q_i}\right\vert}
^{1-\frac1{p_i}} \bigg)
    \bigg(\prod_{i=r+1}^l b_i(x)\bigg) h_j^{(k,1)}(x),
    \end{equation}
    where $b_i(x)$ are the same as above
    and
    $$
    \begin{aligned}
    h_j^{(k,1)}(x) &=
    (2^j\ell(Q_k))^{N_k}{\left\vert{Q_k}\right\vert}^{-1}
    \sum_{{\left\vert{\alpha}\right\vert}=N_k}
    \int_{Q_k}\Big\{\int_0^1\left<2^jx_{c_k,y_k}^t\right>^{s_k}\\
    &\times\norm{
    \prod_{\substack{i=1\\i\ne k}}^{l}\left<y_i\right>^{s_i}
    {\partial^{\alpha}_{k}K_j(y_1,\ldots,y_{k-1},
    2^jx_{c_k,y_k}^t,y_{k+1},\ldots,y_m)}}_{L^2
    (dy_1  \cdots \widehat{dy_k} \cdots dy_m)} dt\Big\}\ dy_k
    \end{aligned}
    $$
    ($\ell (Q_k)$ denotes the sidelength of the cube $Q_k$).
Minkowski's inequality and Lemma \ref{lem:LInfL2} imply that
    $$
    \norm{h_j^{(k,1)}}_{L^2}\lesssim
A 2^{-\frac{jn}{2}}(2^j\ell(Q_k))^{N_k}.
    $$
    Combining inequalities \eqref{ali:H10Func} and \eqref{equ:H1NFunc},
we obtain
    \begin{equation}\begin{split} \label{equ:GjQ1}
     & \bigg(\prod_{i=1}^r\left<2^j(x-c_i)\right>^{s_i}\bigg)
     {\left\vert{g_j(x)}\right\vert}
      \\
    & \hspace{1in} \lesssim 2^{jrn}
    \bigg(\prod_{i=1}^r{\left\vert{Q_i}\right\vert}^{1-\frac1{p_i}}\bigg)
    \bigg(\prod_{i=r+1}^lb_i(x)\bigg)
     \min\left\{h_j^{(k,0)}(x),h_j^{(k,1)}(x)\right\}
     \end{split}\end{equation}
  for all $1\le k\le r.$
  The inequalities in \eqref{equ:GjQ1} imply that
    \begin{equation}\begin{split}\label{equ:GjQ2}
   & {\left\vert{g_j(x)}\right\vert}  \\
  &   \le 2^{jrn}
  \bigg(\prod_{i=1}^r{\left\vert{Q_i}\right\vert}^{1-\frac1{p_i}}
  \left<2^j(x-c_i)\right>^{-s_i}\bigg)
  \bigg(\prod_{i=r+1}^l b_i(x)\bigg)
  \min_{1\le k\le r}\left\{h_j^{(k,0)}(x),h_j^{(k,1)}(x)\right\}
    \end{split}\end{equation}
    for all $x\in (\bigcap_{i=r+1}^lQ_i^*)\setminus (\bigcup_{i=1}^rQ_i^*)$.

    Now we need to construct functions $u^k_j$ $(1\le k \le r)$ such that
$$\abs{g_j(x)}
\lesssim A\prod_{k=1}^ru^k_j(x)\prod_{i=r+1}^lb_i(x)$$
for all  $x\in (\bigcap_{i=r+1}^lQ_i^*)\setminus (\bigcup_{i=1}^rQ_i^*)$ and that
    $\norm{\sum_{j}u^k_j}_{L^{p_k}}\lesssim 1$.
Then the lemma follows by taking $b_k=\sum_{j}u^k_j$ ($1\le k\le r$).

    For this, we choose $\lambda_{k}$, $1\le k\le r$, such that
    $$
    0\le\lambda_k<\frac12,
    \quad
    \frac{s_k}n>\frac1{p_k}-\frac12+\lambda_k,
    \quad
    \sum_{k=1}^r\lambda_k= \dfrac{r-1}2.
    $$
This is possible since \eqref{equ:SIndexCond} implies
    $$
    \sum_{k=1}^r\min\Big\{\frac12,\frac{s_k}{n}-\frac{1}{p_k}+\frac12\Big\}>\frac{r-1}2.
    $$
We set
    $
    \alpha_{k}= \frac1{p_k}-\frac12+\lambda_k$ and
    $\beta_{k}= 1-2\lambda_k $.
    Then we have $\alpha_k >0$, $\beta_k >0$, and $\sum_{k=1}^{r}\beta_k =1$.
 We set
    $$
    u^k_j = A^{-\beta_k}2^{jn}{\left\vert{Q_k}\right\vert}^{1-\frac1{p_k}}
    \left<2^j(\cdot-c_k)\right>^{-s_k}\chi_{(Q_k^*)^c}
 \min\left\{h_j^{(k,0)},h_j^{(k,1)}\right\}^{\beta_k},\quad 1\le k\le r .
    $$
    Then, from \eqref{equ:GjQ2}, it is easy to see that
    $$
    g_j(x)\lesssim A\prod_{k=1}^ru^k_j(x)\prod_{i=r+1}^lb_i(x)
    $$
     for all  $x\in (\cap_{i=r+1}^lQ_i^*)\setminus (\cup_{i=1}^rQ_i^*).$
     It remains to check that
    $\sum_{j}\int_{\mathbb R^n}{ |{u^k_j(x)} |}^{p_k}dx\lesssim 1.$
    Since $\frac1{p_k} = \alpha_k+\frac{\beta_k}2,$
    H\"older's inequality gives
    \begin{align*}
    \norm{u^k_j}_{L^{p_k}}
    \le A^{-\beta_k}2^{jn}
    {\left\vert{Q_k}\right\vert}^{ 1- \frac1{p_k}}
    \norm{\left<2^j(\cdot-c_k)\right>^{-s_k}\chi_{(Q_k^*)^c}}_{L^{\frac{1}{\alpha_k}}}
    \norm{ \min\left\{h_j^{(k,0)},h_j^{(k,1)}\right\} ^{\beta_k}}_{L^{\frac{2}{\beta_k}}}.
    \end{align*}
Since $\frac{s_k}{\alpha_k}>n,$ we have
    $$
    \norm{\left<2^j(\cdot-c_k)\right>^{-s_k}\chi_{(Q_k^*)^c}}
    _{L^{1/\alpha_k}}
    \approx 2^{-jn\alpha_k}
    \min\left\{1,(2^j\ell(Q_k))^{\alpha_kn-s_k}\right\}.
    $$
The estimates of $L^2$-norms of $h_j^{(k,0)}$ and $h_j^{(k,1)}$ given above
imply
    $$
    \begin{aligned}
    \norm{
    \left(
    \min\left\{h_j^{(k,0)},h_j^{(k,1)}\right\}
    \right)^{\beta_k}}_{L^{2/\beta_k}}
    \le
    &\min\left\{
    \norm{h_j^{(k,0)}}_{L^2}^{\beta_k},
    \norm{h_j^{(k,1)}}_{L^2}^{\beta_k}
    \right\}
    \\
    \lesssim
    & \Big(
    A2^{-jn/2}\min\left\{1,(2^j\ell(Q_k))^{N_k}\right\}
    \Big)^{\beta_k}.
    \end{aligned}
    $$
    Therefore
    $$
    \begin{aligned}
    \norm{u^k_j}_{L^{p_k}}\le
    &
    2^{jn}{\left\vert{Q_k}\right\vert}^{1-\frac1{p_k}}
    2^{-jn(\alpha_k + \beta_k/2)}
    \min\left\{   1,  (2^j\ell(Q_k))^{\alpha_kn-s_k}   \right\}
    \min\left\{   1,   (2^j\ell(Q_k))^{N_k\beta_k}   \right\}
    \\
    = &
\begin{cases}
{  (2^{j}\ell (Q_k) )^{n - n/p_k + N_k \beta_k}  }
&
{ \quad \text{if} \quad  2^{j}\ell (Q_k) \le 1  }
\\
{  (2^{j}\ell (Q_k) )^{n - n/p_k + \alpha_k n - s_k}  }
& {  \quad  \text{if}  \quad   2^{j}\ell (Q_k) > 1 }.
\end{cases}
\end{aligned}
$$
This inequality is enough to establish what we needed
    $
    \sum_{j\in\mathbb Z}\int_{\mathbb R^n}{\left\vert{u^k_j(x)}\right\vert}^{p_k}dx\lesssim 1.
    $
The proof of Lemma \ref{lem:KeyLem1} is complete.
  \end{proof}

\begin{proof}[Proof of Lemma \ref{lem:KeyLem2}]
We use the following notations:
\begin{align*}
&\one = \{1,\ldots, l\},
\quad
\two = \{l+1, \ldots, \rho\},
\quad
\three = \{\rho+1, \ldots, m\},
\\
&
A=\{1,\ldots, m\}=\one \cup \two \cup \three.
\end{align*}
Recall that we are assuming $\one\neq \emptyset$ and
$\two \neq \emptyset$
(the set $\three$ might be empty).
For a subset $B=\{i_1, \ldots , i_k\}\subset A$,
we write $y_{B}=(y_{i_1}, \dots,  y_{i_k})$
and $dy_{B}=dy_{i_1}\cdots dy_{i_k}$.
We take a smooth function $\varphi$ on ${\mathbb R}^{n}$ such that
$
\mathrm{supp}\, \varphi \subset
\{\xi \in {\mathbb R}^{n}\mid 4^{-1} a<\abs{\xi}<4\}
$
and $\varphi (\xi)=1$ on $2^{-1}a\le \abs{\xi}\le 2$,
where $a$ is the constant in the definition of $V(k_1, k_2)$,
and define $\Delta_j$, $j\in \mathbb Z$, by \eqref{equ:Deltaj}.
We set $s=\min \{s_1, \dots, s_m\}$ and
take a number $q$ such that
$$
\max \{1,\, n/s \}
<q<2;
$$
this is possible since $s_1, \dots, s_m>n/2$.

Let $a_i$ ($i\in \one$) and $f_i$ ($i\in \two \cup \three$)
be functions as mentioned in the lemma.
Without loss of generality,
we may assume $\norm{f_i}_{L^\infty}=1$ for $i \in \three$.
We use the decomposition
\eqref{equ:SigmaDecomposition}
and write
\begin{align*}
&g= T_{\sigma}(a_1, \dots, a_l, f_{l+1}, \dots, f_{m}),
\\
&g_{j}= T_{\sigma_{j}}(a_1, \dots, a_l, f_{l+1}, \dots, f_{m}).
\end{align*}
Thus $g=\sum_{j\in \mathbb Z}g_j$.

To prove the pointwise estimate \eqref{equ:pointwiseFirstCase},
we divide ${\mathbb R}^{n}$ as
${\mathbb R}^{n}=\bigcup_{J\subset \one}E_J$,
where $J$ runs all subsets of $\one$ and $E_J$ is defined by
$$
E_J = \bigcap_{i\in J}(Q_i^{\ast})^{c}
\cap \bigcap_{i\in \one \setminus J} Q_i^{\ast}.
$$
In order to prove \eqref{equ:pointwiseFirstCase}, it is
sufficient to construct functions
$b^{J}_{i}$ ($i\in \one$) and $\widetilde{f}^{J}_{i}$ ($i\in \two$),
for each $J\subset \one$, such that
\begin{equation}\label{equ:009}
\abs{g(x)} \chi_{E_J}(x)
\lesssim
b^{J}_{1}(x) \dots
b^{J}_{l}(x)
\widetilde{f}^{J}_{l+1}(x)
\dots
\widetilde{f}^{J}_{\rho}(x),
\end{equation}
where
the function $b^{J}_{i}$ depends only on
$m$, $n$, $(s_i)_{i\in A}$, $(p_i)_{i\in A}$,
$\sigma$, $J$, $i$, $a_i$, and $(f_i)_{i\in \three}$;
the function $\widetilde{f}^{J}_i$
depends only on
$m$, $n$, $(s_i)_{i\in A}$,
$J$, $i$, $f_i$, and $(f_i)_{i\in \three}$;
and they satisfy the estimates
\begin{align}
&
\big\|{b^{J}_{i}}\big\|_{L^{p_i}}\lesssim 1,
\label{equ:100}
\\
&
\big\|{\widetilde{f}^{J}_i}\big\|_{L^{2}}\lesssim \norm{f_i}_{L^2}.
\label{equ:101}
\end{align}
In fact, if this is proved, then
the desired functions can be obtained by
$b_{i}=\sum_{J\subset \one} b^{J}_{i}$
and
$\widetilde{f}_{i}=\sum_{J\subset \one} \widetilde{f}^{J}_i$.

First, we shall prove the estimate \eqref{equ:009} for $J=\emptyset$,
$E_{\emptyset}= Q_{1}^{\ast}\cap \dots \cap Q_{l}^{\ast}$.
The argument to be given below will show the estimate \eqref{equ:009}
with some combination of the following choices of
$b^{\emptyset}_{i}$ and
$\widetilde{f}^{\emptyset}_{i}$:
\begin{align}
&
b^{\emptyset}_{i}(x)=
M_{q}(a_i)(x)
\chi_{Q_{i}^{\ast}}(x),
\label{equ:011}
\\
&
b^{\emptyset}_{i}(x)=
\bigg(
\sum_{j\in \mathbb Z}
M_{q}(\Delta_{j} a_i)(x)^2
\bigg)^{1/2}
\chi_{Q_{i}^{\ast}}(x),
\label{equ:010}
\\
&
b^{\emptyset}_{i}(x)=
\bigg(
\sum_{j\in \mathbb Z}
(\zeta_{j}\ast \abs{a_i}^{q})(x)^{2/q}
(\zeta_{j}\ast \abs{\Delta_{j}f_{k}}^{q})(x)^{2/q}
\bigg)^{1/2}
\chi_{Q_{i}^{\ast}}(x),
\quad k\in \three,
\label{equ:012}
\\
&
\widetilde{f}^{\emptyset}_{i}(x)=
M_{q}(f_i)(x),
\label{equ:014}
\\
&
\widetilde{f}^{\emptyset}_{i}(x)=
\bigg(
\sum_{j\in \mathbb Z}
M_{q}(\Delta_{j} f_i)(x)^2
\bigg)^{1/2},
\label{equ:013}
\\
&
\widetilde{f}^{\emptyset}_{i}(x)=
\bigg(
\sum_{j\in \mathbb Z}
(\zeta_{j}\ast \abs{f_i}^{q})(x)^{2/q}
(\zeta_{j}\ast \abs{\Delta_{j}f_{k}}^{q})(x)^{2/q}
\bigg)^{1/2},
\quad k\in \three,
\label{equ:015}
\end{align}
where $\zeta_{j}(x)=2^{jn}(1+\abs{2^j x})^{-sq}$ is the function
in Lemma \ref{lem:Tomita} and
$M_{q}$ denotes the maximal operator defined by
$$
M_{q}(f)(x)
=\sup_{r>0}
\bigg(\frac{1}{r^{n}}
\int_{\abs{x-y}<r}
\abs{f(y)}^{q}\,
dy
\bigg)^{1/q}.
$$
The above functions $b^{\emptyset}_{i}$ and
$\widetilde{f}^{\emptyset}_{i}$
depend on other things as mentioned in the lemma.
We shall see that they also satisfy the estimates
\eqref{equ:100} and \eqref{equ:101}.
For $\widetilde{f}^{\emptyset}_{i}$ given by \eqref{equ:014}
or \eqref{equ:013}, the $L^2$-boundedness of $M_{q}$, $q<2$,
and Lemma \ref{lem:CarlesonMeasure} \eqref{equ:aaa}
give the $L^2$-estimate \eqref{equ:101}.
For $\widetilde{f}^{\emptyset}_{i}$ given by \eqref{equ:015},
Lemma \ref{lem:CarlesonMeasure} \eqref{equ:ccc} yields the same
$L^2$-estimate since
$\norm{f_k}_{BMO}
\lesssim \norm{f_{k}}_{L^\infty}=1$ for $k\in \three$.
For
$b^{\emptyset}_{i}$ given by \eqref{equ:011},
the $L^2$-estimate
$\norm{M_q (a_i)}_{L^2}\lesssim \norm{a_i}_{L^2}$
and H\"older's inequality give the estimate \eqref{equ:100}:
\begin{equation*}
\big\|{b^{\emptyset}_{i}}\big\|_{L^{p_i}}
\le
\norm{
M_{q}(a_i)
}_{L^2}
\abs{Q_{i}^{\ast}}^{1/p_i -1/2}
\lesssim
\norm{a_i}_{L^2}
\abs{Q_{i}}^{1/p_i -1/2}
\le 1.
\end{equation*}
For $b^{\emptyset}_{i}$ given by \eqref{equ:010}
or \eqref{equ:012}, the same estimate is proved in a similar way.

Now we shall divide the proof of \eqref{equ:009} for $J=\emptyset$
into the following
six cases, (1)--(6), depending on the indices
$k_1$ and $k_2$ involved in the assumption
\eqref{equ:support}.

(1) $k_1,k_2\in \one$.
In this case, without loss of generality,
we assume $\{k_1,k_2\}=\{1,2\}\subset \one$.
Then, by the assumption \eqref{equ:support},
it follows that
$2^{j-1}a\le \abs{\xi_1}\le 2^{j+1}$
and $2^{j-1}a\le \abs{\xi_2}\le 2^{j+1}$
for all $\xi \in \mathrm{supp}\, \sigma_{j}$,
and hence $\varphi (2^{-j}\xi_{1})=\varphi (2^{-j}\xi_{2})=1$
on $\mathrm{supp}\, \sigma_{j}$.
Hence we can write
$$
g_{j}=T_{\sigma_{j}}
(\Delta_{j}a_1, \Delta_{j}a_2, a_3, \dots, a_l,
f_{l+1}, \dots, f_{\rho}, \dots, f_{m}).
$$
Hence, by Lemma \ref{lem:Tomita}, we have the
pointwise estimate
\begin{align*}
\abs{g_j}
\lesssim
&(\zeta_{j}\ast \abs{\Delta_{j}a_1}^q)^{1/q}
(\zeta_{j}\ast \abs{\Delta_{j}a_2}^q)^{1/q}
(\zeta_{j}\ast \abs{a_3}^q)^{1/q}
\cdots
(\zeta_{j}\ast \abs{a_l}^q)^{1/q}
\\
&
\times (\zeta_{j}\ast \abs{f_{l+1}}^q)^{1/q}
\cdots
(\zeta_{j}\ast \abs{f_{\rho}}^q)^{1/q}
\cdots
(\zeta_{j}\ast \abs{f_{m}}^q)^{1/q}
\\
\lesssim
&
M_{q}(\Delta_{j}a_1)
M_{q}(\Delta_{j}a_2)
M_{q}(a_3)
\cdots
M_{q}(a_l)
M_{q}(f_{l+1})
\cdots
M_{q}(f_\rho).
\end{align*}
(Notice that
the inequality
$(\zeta_{j}\ast \abs{f}^q)^{1/q}
\lesssim
M_q (f)$ holds because $sq>n$.)
Summing over $j\in \mathbb Z$ and using the
Cauchy-Schwarz inequality, we obtain
\begin{align*}
\abs{g}
\lesssim
&
\bigg(
\sum_{j\in \mathbb Z}
\{M_{q}(\Delta_{j}a_1)\}^{2}
\bigg)^{1/2}
\bigg(
\sum_{j\in \mathbb Z}
\{M_{q}(\Delta_{j}a_2)\}^{2}
\bigg)^{1/2}
\\
&\times
M_{q}(a_3)
\cdots
M_{q}(a_l)
M_{q}(f_{l+1})
\cdots
M_{q}(f_\rho).
\end{align*}
This implies \eqref{equ:009} for $J=\emptyset$
with
$b^{\emptyset}_{i}$ of \eqref{equ:010} for
$i=1,2$,
with
$b^{\emptyset}_{i}$ of \eqref{equ:011} for
$3\le i\le l$,
and
with
$\widetilde{f}^{\emptyset}_{i}$
of \eqref{equ:014} for
$l+1\le i\le \rho$.

(2) $k_1,k_2\in \two$.
In this case, without loss of generality,
we assume $\{k_1,k_2\}=\{l+1,l+2\}\subset \two$.
Then we can write
$$
g_{j}=T_{\sigma_{j}}
(a_1, \dots, a_l, \Delta_{j}f_{l+1}, \Delta_{j}f_{l+2},
f_{l+3}, \dots, f_{\rho}, \dots, f_{m}).
$$
Hence, by Lemma \ref{lem:Tomita},
\begin{align*}
\abs{g_j}
&\lesssim
M_{q}(a_1)
\cdots
M_{q}(a_l)
M_{q}(\Delta_{j}f_{l+1})
M_{q}(\Delta_{j}f_{l+2})
M_{q}(f_{l+3})
\cdots
M_{q}(f_{\rho}).
\end{align*}
Taking sum over $j\in \mathbb Z$ and using the
Cauchy-Schwarz inequality, we obtain
\begin{align*}
\abs{g}
&
\lesssim
M_{q}(a_1)
\cdots
M_{q}(a_l)
\bigg(
\sum_{j\in \mathbb Z}
\{M_{q}(\Delta_{j}f_{l+1})\}^{2}
\bigg)^{1/2}
\bigg(
\sum_{j\in \mathbb Z}
\{M_{q}(\Delta_{j}f_{l+2})\}^{2}
\bigg)^{1/2}
\\
&\qquad\qquad\times
M_{q}(f_{l+3})
\cdots
M_{q}(f_{\rho}).
\end{align*}
This implies \eqref{equ:009} for $J=\emptyset$
with
$b^{\emptyset}_{i}$ of \eqref{equ:011} for
$1\le i\le l$,
with
$\widetilde{f}^{\emptyset}_{i}$
of \eqref{equ:013} for
$i=l+1, l+2$,
and with
$\widetilde{f}^{\emptyset}_{i}$
of \eqref{equ:014} for
$l+3\le i\le \rho$.

(3) $k_1,k_2\in \three$.
Without loss of generality,
we assume $\{k_1,k_2\}=\{\rho+1,\rho+2\}\subset \three$.
Then $g_j$ can be written as
$$
g_{j}=T_{\sigma_{j}}
(a_1, \dots, a_l, f_{l+1}, \dots , f_{\rho},
\Delta_{j}f_{\rho+1},
\Delta_{j}f_{\rho+2},
f_{\rho+3}, \dots, f_{m})
$$
and Lemma \ref{lem:Tomita} yields
\begin{align*}
\abs{g_j}
\lesssim
&(\zeta_{j}\ast \abs{a_1}^{q})^{1/q}
M_{q}(a_2)
\cdots
M_{q}(a_l)
\\
&\times
(\zeta_{j}\ast \abs{f_{l+1}}^{q})^{1/q}
M_{q}(f_{l+2})
\cdots
M_{q}(f_{\rho})
(\zeta_{j}\ast \abs{\Delta_{j}f_{\rho+1}}^q)^{1/q}
(\zeta_{j}\ast \abs{\Delta_{j}f_{\rho+2}}^q)^{1/q}.
\end{align*}
Taking sum over $j\in \mathbb Z$ and using the
Cauchy-Schwarz inequality, we obtain
\begin{align*}
\abs{g}
\lesssim
&
\bigg(
\sum_{j\in \mathbb Z}
(\zeta_{j}\ast \abs{a_1}^{q})^{2/q}
(\zeta_{j}\ast \abs{\Delta_{j}f_{\rho+1}}^{q}
)^{2/q}
\bigg)^{1/2}
M_{q}(a_2)
\cdots
M_{q}(a_l)
\\
&\times
\bigg(
\sum_{j\in \mathbb Z}
(\zeta_{j}\ast \abs{f_{l+1}}^{q})^{2/q}
(\zeta_{j}\ast \abs{\Delta_{j}f_{\rho+2}}^{q}
)^{2/q}
\bigg)^{1/2}
M_{q}(f_{l+2})
\cdots
M_{q}(f_{\rho}).
\end{align*}
This implies \eqref{equ:009} for $J=\emptyset$
with the following functions:
$b^{\emptyset}_{1}$ is \eqref{equ:012} with $i=1$ and $k=\rho+1$;
$b^{\emptyset}_{i}$ is \eqref{equ:011} for $2\le i\le l$;
$\widetilde{f}^{\emptyset}_{l+1}$ is
\eqref{equ:015} with $i=l+1$ and $k=\rho+2$;
and
$\widetilde{f}^{\emptyset}_{i}$ is
\eqref{equ:014} for
$l+2\le i\le \rho$.

(4) $k_1\in \one$ and $k_2\in \two$.
Without loss of generality,
we assume $k_1=1$ and $k_2=l+1$.
Then
$$
g_{j}=T_{\sigma_{j}}
(\Delta_{j}a_1, a_2, \dots, a_l, \Delta_{j}f_{l+1},
f_{l+2}, \dots, f_{\rho}, \dots, f_{m})
$$
and
Lemma \ref{lem:Tomita} yields
\begin{align*}
\abs{g_j}
&\lesssim
M_{q}(\Delta_{j}a_1)
M_{q}(a_2)
\cdots
M_{q}(a_l)
M_{q}(\Delta_{j}f_{l+1})
M_{q}(f_{l+2})
\cdots
M_{q}(f_{\rho}).
\end{align*}
Taking sum over $j\in \mathbb Z$ and using the
Cauchy-Schwarz inequality, we obtain
\begin{align*}
\abs{g}
\lesssim
&
\bigg(\sum_{j\in \mathbb Z}
\{M_{q}(\Delta_{j}a_{1})\}^{2}
\bigg)^{1/2}
M_{q}(a_2)
\cdots
M_{q}(a_l)
\\
&\times
\bigg(
\sum_{j\in \mathbb Z}
\{M_{q}(\Delta_{j}f_{l+1})\}^{2}
\bigg)^{1/2}
M_{q}(f_{l+2})
\cdots
M_{q}(f_{\rho}).
\end{align*}
This implies \eqref{equ:009} for $J=\emptyset$
with
$b^{\emptyset}_{i}$ of \eqref{equ:010} for
$i=1$,
$b^{\emptyset}_{i}$ of \eqref{equ:011} for
$2\le i\le l$,
with
$\widetilde{f}^{\emptyset}_{i}$
of \eqref{equ:013} for
$i=l+1$,
and with
$\widetilde{f}^{\emptyset}_{i}$
of \eqref{equ:014} for
$l+2\le i\le \rho$.

(5) $k_1\in \two$ and $k_2\in \three$.
Without loss of generality,
we assume $k_1=l+1$ and $k_2=\rho+1$.
Then we have
$$
g_{j}=T_{\sigma_{j}}
(a_1, \dots, a_l, \Delta_{j}f_{l+1},
f_{l+2}, \dots, f_{\rho},
\Delta_{j}f_{\rho+1},
f_{\rho+2},\dots, f_m)
$$
and Lemma \ref{lem:Tomita} yields
\begin{align*}
\abs{g_j}
&\lesssim
(\zeta_{j}\ast \abs{a_1}^{q})^{1/q}
M_{q}(a_2)
\cdots
M_{q}(a_l)
\\
&\times
M_{q}(\Delta_{j}f_{l+1})
M_{q}(f_{l+2})
\cdots
M_{q}(f_{\rho})
(\zeta_{j}\ast \abs{\Delta_{j}f_{\rho+1}}^{q})^{1/q}.
\end{align*}
Taking sum over $j\in \mathbb Z$ and using the
Cauchy-Schwarz inequality, we obtain
\begin{align*}
\abs{g}
\lesssim
&
\bigg(\sum_{j\in \mathbb Z}
(\zeta_{j}\ast \abs{a_{1}}^{q})^{2/q}
(\zeta_{j}\ast \abs{\Delta_{j}f_{\rho+1}}^{q})^{2/q}
\bigg)^{1/2}
M_{q}(a_2)
\cdots
M_{q}(a_l)
\\
&\times
\bigg(
\sum_{j\in \mathbb Z}
\{M_{q}(\Delta_{j}f_{l+1})\}^{2}
\bigg)^{1/2}
M_{q}(f_{l+2})
\cdots
M_{q}(f_{\rho}).
\end{align*}
This implies \eqref{equ:009} for $J=\emptyset$
with the following functions:
$b^{\emptyset}_{1}$ is \eqref{equ:012} with $i=1$ and $k=\rho+1$;
$b^{\emptyset}_{i}$ is \eqref{equ:011} for
$2\le i\le l$;
$\widetilde{f}^{\emptyset}_{i}$
is \eqref{equ:013} for
$i=l+1$;
and
$\widetilde{f}^{\emptyset}_{i}$
is \eqref{equ:014} for
$l+2\le i\le \rho$.

(6) $k_1\in \one$ and $k_2\in \three$.
Without loss of generality,
we assume $k_1=1$ and $k_2=\rho+1$.
Then $g_j$ can be written as
$$
g_{j}=T_{\sigma_{j}}
(\Delta_{j}a_1, a_2, \dots, a_l, f_{l+1},
\dots, f_{\rho},
\Delta_{j}f_{\rho+1},
f_{\rho+2},\dots, f_m)
$$
and Lemma \ref{lem:Tomita} yields
\begin{align*}
\abs{g_j}
&\lesssim
M_{q}(\Delta_{j}a_1)
M_{q}(a_2)
\cdots
M_{q}(a_l)
\\
&\times
(\zeta_{j}\ast \abs{f_{l+1}}^{q})^{1/q}
M_{q}(f_{l+2})
\cdots
M_{q}(f_{\rho})
(\zeta_{j}\ast \abs{\Delta_{j}f_{\rho+1}}^{q})^{1/q}.
\end{align*}
Using the Cauchy-Schwarz inequality, we obtain
\begin{align*}
\abs{g}
\lesssim
&
\bigg(\sum_{j\in \mathbb Z}
\{M_{q}(\Delta_{j} a_{1})\}^{2}
\bigg)^{1/2}
M_{q}(a_2)
\cdots
M_{q}(a_l)
\\
&\times
\bigg(
\sum_{j\in \mathbb Z}
(\zeta_{j}\ast \abs{f_{l+1}}^{q})^{2/q}
(\zeta_{j}\ast \abs{\Delta_{j}f_{\rho+1}}^{q})^{2/q}
\bigg)^{1/2}
M_{q}(f_{l+2})
\cdots
M_{q}(f_{\rho}).
\end{align*}
This implies \eqref{equ:009} for $J=\emptyset$
with the following functions:
$b^{\emptyset}_{i}$ is \eqref{equ:010} for
$i=1$;
$b^{\emptyset}_{i}$ is \eqref{equ:011} for
$2\le i\le l$;
$\widetilde{f}^{\emptyset}_{l+1}$ is
\eqref{equ:015} with
$i=l+1$ and $k=\rho+1$;
and $\widetilde{f}^{\emptyset}_{i}$
is \eqref{equ:014} for
$l+2\le i\le \rho$.
Thus we have proved
\eqref{equ:009} for $J=\emptyset$.

Next we shall prove
\eqref{equ:009} for $J\neq \emptyset$.
Here we will not use the assumption
\eqref{equ:support}.
We fix a nonempty subset $J\subset \one$.
We shall prove that there exist
functions
$u^{J}_{k,j}$, $k\in J$, $j\in \mathbb Z$,
such that
\begin{equation}\label{equ:019}
\abs{g_j(x)}
\chi_{E_{J}}(x)
\lesssim
\prod_{k\in J}
u^{J}_{k,j}(x)\cdot
\prod_{i\in \one \setminus J}
\abs{Q_{i}}^{-1/p_i}
\chi_{Q_{i}^{\ast}}(x)
\cdot
\prod_{i\in \two }
M_{q}(f_{i})(x)
\end{equation}
for all $j\in \mathbb Z$ and all $x\in {\mathbb R}^{n}$;
the function $u^{J}_{k,j}$ depends only on
$m$, $n$, $(s_{i})_{i\in A}$, $(p_{i})_{i\in A}$, $\sigma$,
$J$, $k$, $j$, $N_k$, and $Q_{k}$,
and satisfies the estimate
\begin{equation}\label{equ:020}
\norm{u^{J}_{k,j}}_{L^{p_k}}
\lesssim
\min \{
(2^{j}\ell (Q_{k}))^{\gamma_k}, \,
(2^{j}\ell (Q_{k}))^{-\delta_k}
\},
\end{equation}
where $\gamma_{k}$ and $\delta_{k}$ are
positive constants that
will be given in terms of
$n$, $k$, $J$, $(s_{i})_{i\in J}$, $(p_{i})_{i\in J}$,
and $N_{k}$.
If we have these functions $u^{J}_{k,j}$,
then we have \eqref{equ:009} with the functions
\begin{align*}
&b^{J}_{k}=\sum_{j\in \mathbb Z}u^{J}_{k,j}
\quad\text{for}\quad k\in J,
\\
&b^{J}_{i}=
\abs{Q_{i}}^{-1/p_i}\chi_{Q_{i}^{\ast}}
\quad\text{for}\quad i\in \one \setminus J,
\\
&\widetilde{f}^{J}_{i}=M_{q}(f_i)
\quad\text{for}\quad i\in \two.
\end{align*}
In fact, $b^{J}_{k}$, $k\in J$,
depends only on
$m$, $n$, $(s_{i})_{i\in J}$, $(p_{i})_{i\in J}$,
$\sigma$, $J$, $k$, $N_{k}$, and $Q_{k}$,
and the estimate
\eqref{equ:100} follows from \eqref{equ:020}.
The estimate \eqref{equ:100} for $b^{J}_{i}$ with
$i\in \one \setminus J$ is obvious and
the estimate \eqref{equ:101} for $\widetilde{f}_{i}$
with $i\in \two$ holds by the $L^2$-boundedness of $M_{q}$, $q<2$.
Thus it is sufficient to construct the functions $u^{J}_{k,j}$.

Before we proceed to the construction of $u^{J}_{k,j}$,
we observe that it is sufficient to treat only the
case $j=0$.
In fact, if we have \eqref{equ:019}-\eqref{equ:020} for $j=0$,
then the case of general $j\in \mathbb Z$ can be derived
by the use of the dilation formula
$$
T_{\sigma_{j}}
(f_1,\dots, f_{m})(x)
=
T_{\sigma_{j}(2^{j}\cdot )}
(f_1(2^{-j}\cdot),\dots, f_{m}(2^{-j}\cdot))(2^{j}x)
$$
and by simple computation.

Thus we shall consider $g_{0}(x)$.
Using $K_{0}=(\sigma_{0})^{\vee}$
(the inverse Fourier transform of $\sigma_{0}$),
we write
\begin{equation}\label{equ:021}
g_{0}(x)
=\int_{{\mathbb R}^{mn}}
K_{0}(x-y_1,\dots, x-y_m)
\prod_{i\in \one}
a_{i}(y_i)
\cdot
\prod_{i\in\two \cup \three }
f_{i}(y_i)\,
dy_1\cdots dy_m.
\end{equation}
We write $c_i$ to denote the center of the cube $Q_i$.
Since $\abs{x-y_i}\approx \abs{x-c_i}$ for $x\not\in Q_{i}^{\ast}$
and $y_{i}\in Q_{i}$, from
\eqref{equ:021}
we see that
the following inequalities hold for $x\in E_{J}$:
\begin{align*}
&
\prod_{i\in J}\langle x-c_i \rangle^{s_i}
\cdot \abs{g_{0}(x)}
\\
&
\lesssim
\int_{{\mathbb R}^{mn}}
\prod_{i\in J}\langle x-y_i \rangle^{s_i}
\cdot
\abs{
K_{0}(x-y_1,\dots, x-y_m)
}
\prod_{i\in \one}
\abs{a_{i}(y_i)}
\cdot
\prod_{i\in\two \cup \three }
\abs{f_{i}(y_i)}
\,dy_1\cdots dy_m
\\
&
\le
\int_{{\mathbb R}^{mn}}
\prod_{i\in J}\langle x-y_i \rangle^{s_i}
\cdot
\abs{
K_{0}(x-y_1,\dots, x-y_m)
}
\\
&\quad \quad \quad \quad
\times
\prod_{i\in \one}
\abs{Q_{i}}^{-1/p_i}\chi_{Q_{i}}(y_i)
\cdot
\prod_{i\in\two}
\abs{f_{i}(y_i)}
\,dy_1\cdots dy_m.
\end{align*}
We take a $k\in J$ and estimate the last integral as
\begin{align*}
\le
&\int_{{\mathbb R}^{n}}
\norm{
\prod_{i\in J\cup \two}
\langle x-y_i \rangle^{s_i}
\cdot
K_{0}(x-y_1,\dots, x-y_m)
}
_{
L^{\infty}(y_{J\setminus \{k\}})
L^{1}(y_{\one \setminus J})
L^{q^{\prime}}(y_{\two})
L^{1}(y_{\three})
}
\\
&\times
\norm{
\prod_{i\in J}
\abs{Q_{i}}^{-1/p_i}\chi_{Q_{i}}(y_i)
}_{L^{1}(y_{J\setminus \{k\}})}
\norm{
\prod_{i\in \one \setminus J}
\abs{Q_{i}}^{-1/p_i}\chi_{Q_{i}}(y_i)
}_{L^{\infty}(y_{\one \setminus J})}
\\
&
\times
\norm{
\prod_{i\in \two}
\langle
x-y_i
\rangle^{-s_i}
f_{i}(y_i)
}_{L^{q}(y_{\two})}
\, dy_{k},
\end{align*}
where we used the following notation
for mixed norm and its obvious generalization:
$$
\norm{
F(z_1,z_2)
}_{L^{p}(z_1)L^{q}(z_2)}
=
\left[
\int_{{\mathbb R}^n}
\bigg(\int_{{\mathbb R}^n}
\abs{F(z_1,z_2)}^{p}\, dz_1
\bigg)^{q/p}
\,dz_2
\right]^{1/q}.
$$
Recall that the mixed norms satisfy
\begin{equation}\label{equ:022}
\norm{
F(z_1,z_2)
}_{L^{p}(z_1)L^{q}(z_2)}
\le
\norm{
F(z_1,z_2)
}_{L^{q}(z_2)L^{p}(z_1)}
\quad \text{if}\quad p<q.
\end{equation}
Since $s_{i}>n/2$, the Cauchy-Schwarz inequality
gives
\begin{equation}\label{equ:023}
\norm{
F(x-y_1,\dots, x-y_m)
}_{L^{1}(y_{B})}
\lesssim
\norm{
\prod_{i\in B}\langle x-y_i \rangle ^{s_i}
\cdot
F(x-y_1,\dots, x-y_m)
}_{L^{2}(y_{B})}.
\end{equation}
Now repeated applications of
\eqref{equ:022}, \eqref{equ:023},
and Lemma \ref{lem:LInfL2} yield
\begin{align*}
&
\norm{
\prod_{i\in J\cup \two}
\langle x-y_i \rangle^{s_i}
\cdot
K_{0}(x-y_1,\dots, x-y_m)
}
_{
L^{\infty}(y_{J\setminus \{k\}})
L^{1}(y_{\one \setminus J})
L^{q^{\prime}}(y_{\two})
L^{1}(y_{\three})
}
\\
&
\lesssim
\norm{
\prod_{i\in A}
\langle x-y_i \rangle^{s_i}
\cdot
K_{0}(x-y_1,\dots, x-y_m)
}
_{
L^{\infty}(y_{J\setminus \{k\}})
L^{2}(y_{\one \setminus J})
L^{q^{\prime}}(y_{\two})
L^
{2}(y_{\three})
}
\\
&\lesssim
\norm{
\prod_{i\in A}
\langle x-y_i \rangle^{s_i}
\cdot
K_{0}(x-y_1,\dots, x-y_m)
}
_{
L^{2}(y_{A\setminus \{k\}})
}
\\
&
=
\bigg\Vert
\langle x-y_k \rangle^{s_k}
\prod_{i\in A\setminus \{k\}}
\langle z_i \rangle^{s_i}
\cdot
K_{0}(z_1,\dots, x-y_k,\dots, z_m)
\bigg\Vert
_{
L^{2}(z_{A\setminus \{k\}})
}.
\end{align*}
Since $s_i q>n$ by our choice of $q$, we have
$$
\norm{
\prod_{i\in \two}
\langle
x-y_i
\rangle^{-s_i}
f_{i}(y_i)
}_{L^{q}(y_{\two})}
\lesssim
\prod_{i\in \two}
M_{q}
(f_{i})(x).
$$
Combining the above inequalities, we obtain
the following estimate for $x\in E_{J}$:
\begin{equation}\label{equ:024}
\begin{split}
&\prod_{i\in J}\langle x-c_i \rangle^{s_i}
\cdot \abs{g_{0}(x)}
\\
&\lesssim
h^{(k,0)}(x)
\prod_{i\in J}\abs{Q_i}^{-1/p_i +1}
\cdot
\prod_{i\in \one \setminus J}\abs{Q_i}^{-1/p_i}
\cdot
\prod_{i\in \two }M_{q}(f_i)(x),
\end{split}
\end{equation}
where
\begin{align*}
&h^{(k,0)}(x)
\\
&
=\abs{Q_k}^{-1}\int_{Q_k}
\bigg\Vert
\langle x-y_k \rangle^{s_k}
\prod_{i\in A\setminus \{k\}}
\langle
z_i
\rangle^{s_i}
\cdot
K_{0}(
z_1, \dots , x-y_k, \dots, z_m)
\bigg\Vert
_{L^2(z_{A\setminus \{k\}})}
\, dy_k.
\end{align*}
We have
\begin{align*}
&
\norm{h^{(k,0)}}_{L^2({\mathbb R}^n)}
\\
&
\le
\abs{Q_k}^{-1}\int_{Q_k}
\bigg\Vert
\langle x-y_k \rangle^{s_k}
\prod_{i\in A\setminus \{k\}}
\langle
z_i
\rangle^{s_i}
\cdot
K_{0}(
z_1, \dots , x-y_k, \dots, z_m)
\bigg\Vert
_{L^2(z_{A\setminus \{k\}})L^{2}(x)}
\, dy_k
\\
&
=
\norm{
\prod_{i\in A}
\langle
z_i
\rangle^{s_i}
\cdot
K_{0}(
z_1, \dots, z_m)
}
_{L^2(z_{A})}
=
\norm{\sigma_{0}}_{W^{(s_1,\dots,s_m)}}.
\end{align*}
Thus, by the assumption \eqref{equ:normalize},
\begin{equation}\label{equ:025}
\big\|{h^{(k,0)}}\big\|_{L^2({\mathbb R}^n)}
\le
1.
\end{equation}

On the other hand,
using the vanishing moment condition of $a_k$ and Taylor's formula,
we can write $g_{0}(x)$ as
$$
\begin{aligned}
g_0(x)
 \, =
&\,
\sum_{{\left\vert{\alpha}\right\vert}=N_k}
C_{\alpha}
\int_{\mathbb R^{mn}}
\Big\{
\int_0^1
(1-t)^{N_k-1}
\\
&\times
\partial^{\alpha}_{k}
K_0
\Big(x-y_1,\ldots,x^{t}_{c_k,y_k},
\ldots, x-y_m
\Big)\\
&\times(y_k-c_k)^{\alpha}
a_1(y_1)\cdots a_l(y_l)
f_{l+1}(y_{l+1})\cdots f_m(y_m)
\, dt
\Big\}
\, dy_1\cdots dy_m,
\end{aligned}
$$
where
$\partial^{\alpha}_{k}
K_0 (z_1,\dots, z_m)
=
\partial_{z_k}^{\alpha}K_0 (z_1,\dots, z_m)$
and
$x^{t}_{c_k,y_k}=x-c_k-t(y_k-c_k)$.
Hence the following inequality holds for $x\in E_J$:
\begin{align*}
&
\prod_{i\in J}\langle x- c_i \rangle^{s_i}\cdot
\abs{g_0(x)}
\\
&\lesssim
\sum_{{\left\vert{\alpha}\right\vert}=N_k}
\int_{\mathbb R^{mn}}
\Big\{
\int_0^1
\langle x^{t}_{c_k,y_k}\rangle^{s_k}
\prod_{i\in J\setminus \{k\}}
\langle x- y_i \rangle^{s_i}
\\
&\times
\abs{
\partial^{\alpha}_{k}K_0
\Big(x-y_1,\ldots,x^{t}_{c_k,y_k},
\ldots, x-y_m
\Big)
}
\\
&\times
\ell (Q_k)^{N_k}
\prod_{i\in \one}\abs{Q_i}^{-1/p_i}
\chi_{Q_i}(y_i)
\cdot
\prod_{i\in \two}
\abs{f_i (y_i)}
\, dt
\Big\}
\, dy_1\cdots dy_m.
\end{align*}
Using this inequality and arguing
in the same way as before, we
obtain the following estimate for $x\in E_{J}$:
\begin{equation}\label{equ:026}
\begin{split}
&\prod_{i\in J}\langle x-c_i \rangle^{s_i}
\cdot \abs{g_{0}(x)}
\\
&\lesssim
h^{(k,1)}(x)
\prod_{i\in J}\abs{Q_i}^{-1/p_i +1}
\cdot
\prod_{i\in \one \setminus J}\abs{Q_i}^{-1/p_i}
\cdot
\prod_{i\in \two }M_{q}(f_i)(x),
\end{split}
\end{equation}
where
\begin{align*}
&h^{(k,1)}(x)
=\abs{Q_k}^{-1+N_k/n}
\sum_{\abs{\alpha}=N_k}
\int_{\substack{
0<t<1
\\
y_k\in Q_k
}}
\\
&
\bigg\Vert
\langle x^{t}_{c_k,y_k} \rangle^{s_k}
\prod_{i\in A\setminus \{k\}}
\langle
z_i
\rangle^{s_i}
\cdot
\partial_{k}^{\alpha}K_{0}(
z_1, \dots , x^{t}_{c_k,y_k}, \dots, z_m)
\bigg\Vert
_{L^2(z_{A\setminus \{k\}})}
\, dtdy_k.
\end{align*}
Using Lemma \ref{lem:LInfL2}, we obtain
\begin{equation}\label{equ:027}
\norm{h^{(k,1)}}_{L^2({\mathbb R}^n)}
\lesssim
\abs{Q_k}^{N_k/n}.
\end{equation}

From the two estimates \eqref{equ:024} and \eqref{equ:026}, we
obtain
\begin{align*}
\abs{g_{0}(x)}
\lesssim
&\prod_{i\in J}\langle x-c_i \rangle^{-s_i}
\abs{Q_i}^{-1/p_i +1}
\cdot
\prod_{i\in \one \setminus J}\abs{Q_i}^{-1/p_i}
\cdot
\prod_{i\in \two }M_{q}(f_i)(x)
\\
&\times
\min \{h^{(k,0)}(x),\, h^{(k,1)}(x)\}
\end{align*}
for all $x\in E_J$ and all $k\in J$.
We take positive numbers $(\beta_{k})_{k\in J}$
satisfying $\sum_{k\in J}\beta_k =1$ and
take a geometric mean of the above estimates to obtain
\begin{equation*}
\abs{g_{0}(x)}
\chi_{E_J}(x)
\lesssim
\prod_{k\in J}
u^{J}_{k}(x)
\cdot
\prod_{i\in \one \setminus J}\abs{Q_i}^{-1/p_i}
\chi_{Q_i^{\ast}}(x)
\cdot
\prod_{i\in \two }M_{q}(f_i)(x),
\end{equation*}
where
$$
u^{J}_{k}(x)
=
\langle x-c_k \rangle^{-s_k}
\abs{Q_k}^{-1/p_k +1}\chi_{(Q_k^{\ast})^{c}}(x)
\left(
\min \{h^{(k,0)}(x),\, h^{(k,1)}(x)\}
\right)^{\beta_k}.
$$
We choose $\beta_k$, $k\in J$, so that we have
$$
\beta_k >0,
\quad
\frac{s_k}{n}>\frac{1}{p_k}- \frac{\beta_{k}}{2},
\quad \sum_{k\in J}\beta_k =1.
$$
This is possible since
$1/2 > \sum_{k\in J}\max \{0, {1}/{p_k} - {s_k}/{n}\}$
by virtue of our condition \eqref{equ:SIndexCond}.
If we write ${1}/{p_k} - {\beta_k}/{2}=1/r_k$, then
$r_k>0$ and H\"older's
inequality gives
\begin{align*}
\norm{u^{J}_{k}}_{L^{p_k}}
\le
&
\norm{
\langle x-c_k \rangle ^{-s_k}
\abs{Q_k}^{-1/p_k +1}\chi_{(Q_k^{\ast})^{c}}(x)
}_{L^{r_k}}
\\
&
\qquad \qquad \times
\norm{
\left(
\min \{h^{(k,0)}(x),\, h^{(k,1)}(x)\}
\right)^{\beta_k}
}_{L^{2/\beta_k}}.
\end{align*}
Since $s_k r_k>n$, we have
\begin{equation*}
\norm{
\langle x-c_k \rangle ^{-s_k}
\abs{Q_k}^{-1/p_k +1}\chi_{(Q_k^{\ast})^{c}}(x)
}_{L^{r_k}}
\approx
\begin{cases}
{\abs{Q_k}^{-1/p_k +1}} & {\quad\text{if $\abs{Q_k}\le 1$}\quad }
\\
{\abs{Q_k}^{-1/p_k +1-s_k/n + 1/r_k}}
& {\quad\text{if $\abs{Q_k}> 1$. }\quad }
\end{cases}
\end{equation*}
By \eqref{equ:025} and \eqref{equ:027}, we have
\begin{align*}
\norm{
\left(
\min \{h^{(k,0)}(x),\, h^{(k,1)}(x)\}
\right)^{\beta_k}
}_{L^{2/\beta_k}}
&\le
\min \left\{
\norm{h^{(k,0)}}_{L^2}^{\beta_k},\,
\norm{h^{(k,1)}}_{L^2}^{\beta_k}
\right\}
\\
&\lesssim
\begin{cases}
{\abs{Q_k}^{N_k \beta_k/n}} &
{\quad\text{if $\abs{Q_k}\le 1$}\quad }
\\
{1}
& {\quad\text{if $\abs{Q_k}> 1$.}\quad }
\end{cases}
\end{align*}
Thus
$$
\norm{
u^{J}_{k}
}_{L^{p_k}}
\lesssim
\begin{cases}
{\abs{Q_k}^{N_k \beta_k/n - 1/p_k +1}}
& {\quad\text{if $\abs{Q_k}\le 1$}\quad }
\\
{\abs{Q_k}^{-1/p_k +1-s_k/n + 1/r_k}}
& {\quad\text{if $\abs{Q_k}> 1$,}\quad }
\end{cases}
$$
which implies \eqref{equ:020} for $j=0$
with
$\gamma_k= N_k \beta_k - n/p_k +n$ and
$\delta_k=n/p_k -n+s_k - n/r_k$.
We have $\gamma_k>0$ since $N_k$ is sufficiently large
and $\delta_{k}>0$ since
$\delta_{k}=n \beta_k/2 -n +s_k\ge n \beta_k/2-n/p_k +s_k >0$
by our choice of $\beta_k$.
This completes the proof of Lemma \ref{lem:KeyLem2}.
\end{proof}

\begin{proof}[Proof of Lemma \ref{lem:KeyLem3}]
Since the proof is similar to that of Lemma \ref{lem:KeyLem2},
we shall briefly indicate only the key points.
We   use the same notation as in the proof of
Lemma \ref{lem:KeyLem2}.
We also write
$$
G(x) =GT_{\sigma}(a_1, \dots, a_l, f_{l+1}, \dots, f_m)(x)
=\bigg(
\sum_{j\in \mathbb Z}
\abs{g_j (x)}^{2}
\bigg)^{1/2}.
$$
It is sufficient to prove the estimate
\begin{equation}\label{equ:028}
G(x) \chi_{E_J}(x)
\lesssim
\prod_{i\in \one} b^{J}_{i}(x)\cdot
\prod_{i\in \two} \widetilde{f}_{i}(x)
\end{equation}
for each subset $J\subset \one$, where
$b^{J}_{i}$ and $\widetilde{f}_{i}$ have the
same properties as in \eqref{equ:009}.

First we consider the case $J=\emptyset$,
$E_{\emptyset}=Q_{1}^{\ast}\cap \cdots \cap Q_{l}^{\ast}$.
We divide the proof into the following three cases, (1)--(3),
depending on the index $k_1$ involved in the assumption
\eqref{equ:support} with $k_2=m+1$.

(1) $k_1\in \one$.
Without loss of generality,
we assume $k_1=1$.
We can write
$$
g_{j}
=T_{\sigma_{j}}
(\Delta_{j}a_1, a_2, \dots, a_l, f_{l+1}, \dots, f_{\rho}, \dots, f_{m}).
$$
By Lemma \ref{lem:Tomita}, we have
\begin{equation*}
\abs{g_j}
\lesssim
M_{q}(\Delta_{j}a_1)
M_{q}(a_2)
\cdots
M_{q}(a_l)
M_{q}(f_{l+1})
\cdots
M_{q}(f_\rho).
\end{equation*}
Hence
\begin{equation*}
G
\lesssim
\bigg(
\sum_{j\in \mathbb Z}
\{M_{q}(\Delta_{j}a_1)\}^{2}
\bigg)^{1/2}
M_{q}(a_2)
\cdots
M_{q}(a_l)
M_{q}(f_{l+1})
\cdots
M_{q}(f_\rho).
\end{equation*}
Thus we obtain \eqref{equ:028} for $J=\emptyset$ with
\begin{align*}
&
b^{\emptyset}_{1}=
\bigg(
\sum_{j\in \mathbb Z}
\{M_{q}(\Delta_{j}a_1)\}^{2}
\bigg)^{1/2}
\chi_{Q_1^{\ast}},
\\
&
b^{\emptyset}_{i}=
M_{q}(a_i)
\chi_{Q_i^{\ast}}
\quad \text{for}\quad 2\le i\le l,
\\
&
\widetilde{f}^{\emptyset}_{i}=
M_{q}(f_i)
\quad \text{for}\quad l+1\le i\le \rho.
\end{align*}

(2) $k_1\in \two$.
Without loss of generality,
we assume $k_1=l+1$.
We can write
$$
g_{j}
=T_{\sigma_{j}}
(a_1, \dots, a_l, \Delta_{j}f_{l+1}, f_{l+2}, \dots, f_{\rho}, \dots, f_{m}).
$$
By Lemma \ref{lem:Tomita}, we have
\begin{equation*}
\abs{g_j}
\lesssim
M_{q}(a_1)
\cdots
M_{q}(a_l)
M_{q}(\Delta_{j}f_{l+1})
M_{q}(f_{l+2})
\cdots
M_{q}(f_\rho).
\end{equation*}
Hence
\begin{equation*}
G
\lesssim
M_{q}(a_1)
\cdots
M_{q}(a_l)
\bigg(
\sum_{j\in \mathbb Z}
\{M_{q}(\Delta_{j}f_{l+1})\}^{2}
\bigg)^{1/2}
M_{q}(f_{l+2})
\cdots
M_{q}(f_\rho).
\end{equation*}
Thus we obtain \eqref{equ:028} for $J=\emptyset$ with
\begin{align*}
&
b^{\emptyset}_{i}=
M_{q}(a_i)
\chi_{Q_i^{\ast}}
\quad \text{for}\quad 1\le i\le l,
\\
&
\widetilde{f}^{\emptyset}_{l+1}=
\bigg(
\sum_{j\in \mathbb Z}
\{M_{q}(\Delta_{j}f_{l+1})\}^{2}
\bigg)^{1/2},
\\
&
\widetilde{f}^{\emptyset}_{i}=
M_{q}(f_i)
\quad \text{for}\quad l+2\le i\le \rho.
\end{align*}

(3) $k_1\in \three$.
Without loss of generality,
we assume $k_1=\rho+1$.
We can write
$$
g_{j}
=T_{\sigma_{j}}
(a_1, \dots, a_l, f_{l+1}, \dots,
f_{\rho}, \Delta_{j}f_{\rho+1}, \dots, f_{m}).
$$
Lemma \ref{lem:Tomita} yields
\begin{equation*}
\abs{g_j}
\lesssim
(\zeta_{j}\ast \abs{a_1}^q )^{1/q}
M_{q}(a_2)\cdots
M_{q}(a_l)
M_{q}(f_{l+1})
\cdots
M_{q}(f_\rho)
(\zeta_{j}\ast \abs{\Delta_{j}f_{\rho+1}}^q )^{1/q}.
\end{equation*}
Hence
\begin{equation*}
G
\lesssim
\bigg(
\sum_{j\in \mathbb Z}
(\zeta_{j}\ast \abs{a_1}^q )^{2/q}
(\zeta_{j}\ast \abs{\Delta_{j}f_{\rho+1}}^q )^{2/q}
\bigg)^{1/2}
M_{q}(a_2)\cdots M_{q}(a_l)
M_{q}(f_{l+1})
\cdots
M_{q}(f_\rho).
\end{equation*}
Thus we obtain \eqref{equ:028} for $J=\emptyset$ with
\begin{align*}
&
b^{\emptyset}_{1}=
\bigg(
\sum_{j\in \mathbb Z}
(\zeta_{j}\ast \abs{a_1}^q )^{2/q}
(\zeta_{j}\ast \abs{\Delta_{j}f_{\rho+1}}^q )^{2/q}
\bigg)^{1/2}\chi_{Q_1^{\ast}},
\\
&
b^{\emptyset}_{i}=
M_{q}(a_i)
\chi_{Q_i^{\ast}}
\quad \text{for}\quad 2\le i\le l,
\\
&
\widetilde{f}^{\emptyset}_{i}=
M_{q}(f_i)
\quad \text{for}\quad l+1\le i\le \rho.
\end{align*}

Finally we prove \eqref{equ:028} for $J\neq \emptyset$.
The proof is immediate.
Observe that
the estimate
of $g_{j}(x)$ on $E_{J}$, $J\neq \emptyset$,
given in the latter half of
the proof of Lemma \ref{lem:KeyLem2}
holds in the present case as well,
since we did not use the assumption
\eqref{equ:support} in that argument.
Also observe that there we have actually proved the estimate
\begin{equation*}
\sum_{j\in \mathbb Z} \abs{g_j (x)}
\chi_{E_{J}}(x)
\lesssim
b^{J}_{1}(x)\cdots
b^{J}_{l}(x)
\widetilde{f}^{J}_{l+1}(x)
\cdots
\widetilde{f}^{J}_{\rho}(x)
\end{equation*}
for $J\neq \emptyset$.
Thus
the estimate \eqref{equ:028} for $J\neq \emptyset$
also holds since
$$
G(x)
=\bigg(\sum_{j\in \mathbb Z} \abs{g_j (x)}^{2}
\bigg)^{1/2}
\le
\sum_{j\in \mathbb Z} \abs{g_j (x)}.
$$
This completes the proof of Lemma \ref{lem:KeyLem3}.
\end{proof}

\section{The space $L^1$ and weak type estimates}\label{section:WeakEstimate}
In this section, we   prove that if we replace $H^1$ by $L^1$, then
we obtain the weak type estimate for $T_{\sigma}$
under the same regularity assumption on the multipliers.
Precisely, we   prove the following theorem.

\begin{thm}\label{thm:WeakEstimate}
Let $s_1, \dots, s_m$, $p_1,\ldots,p_m$, and $p$ satisfy
the same assumptions as in Theorem \ref{thm:General}.
Define $X_i$, $i=1,\dots, m$, by
$X_i=H^{p_i}$ if $p_i \neq 1$ and
$X_i=L^1$ if $p_i = 1$.
Then
\begin{equation}\label{equ:TSigmWeakBound}
\norm{T_{\sigma}}_{X_1\times\cdots\times X_m
\longrightarrow  L^{(p,\infty)}}
\lesssim
\sup_{j\in\mathbb Z}
{\left\Vert\sigma(2^j\cdot)\widehat\psi\,
\right\Vert}_{W^{(s_1,\ldots,s_m)}}.
\end{equation}
The conditions given above are optimal
in the sense that if \eqref{equ:TSigmWeakBound}
holds
then
we must have
$s_1, \dots, s_m \ge n/2$ and \eqref{equ:SIndexCondGe}
for every nonempty subset $J\subset\left\{1,2,\ldots,m\right\}$.
\end{thm}

The proof depends on the following lemma, which is
a slight generalization of the remark given in Stein \cite[5.24]{Stein}.

\begin{lem}\label{lem:Stein}
Let $p_0, p_1, q_0, q_1, r$ satisfy
$n/(n+1)<p_0 < 1 < p_1 < \infty$,
$0<q_0<r< q_1<\infty$,
and $1/p_0  - 1/q_0=1/p_1 - 1/q_1= 1 - 1/r $.
Let $T$ be a linear mapping of
$L^1 ({\mathbb R}^n)$ into $\mathcal{M}({\mathbb R})^n$,
the space of all measurable functions on ${\mathbb R}^n$.
Assume the estimates
\begin{align}
&\norm{Tf}_{L^{(q_0,\infty)}}
\le M_0 \norm{f}_{H^{p_0}},
\label{equ:WeakHp0Bound}
\\
&\norm{Tf}_{L^{(q_1,\infty)}}
\le M_1 \norm{f}_{L^{p_1}}
\label{equ:WeakLp1Bound}
\end{align}
holds for all $f\in L^1 ({\mathbb R}^{n})$ with
the right hand sides finite, where $M_0$ and $M_1$ are
positive constants.
Then
\begin{equation*}
\norm{Tf}_{L^{(r,\infty)}}
\le C M_{0}^{1-\theta}M_{1}^{\theta}
\norm{f}_{L^1}
\end{equation*}
for all $f\in L^1 (\mathbb R^n)$, where
$C$ is a constant
depending only on
$p_0, p_1, q_0, q_1, r$, and $n$,
and $\theta$ is given by
$1=(1-\theta)/p_0 + \theta / p_1$.
\end{lem}

\begin{proof}
Let $f\in L^1 (\mathbb R^n)$ and
we assume $\norm{f}_{L^1}=1$.
Let $0<\lambda < \infty$ be given.
We apply the Calder\'on-Zygmund decomposition to $f$
at height $A \lambda^{r}$,
where $A$ is a positive constant to be determined later.
Thus we obtain a family of disjoint cubes $\{Q_j\}$ such that
\begin{align*}
&
A \lambda^{r} <
\frac{1}{\abs{Q_j}}
\int_{Q_j} \abs{f(x)}\, dx
\le 2^{n} A \lambda^{r},
\\
&
\abs{f(x)} \le A \lambda^{r}
\quad \text{for a.\,e.} \quad x \not\in \bigcup_{j}Q_j,
\\
&
\sum_{j}
\abs{Q_j}
\le \left(A \lambda^{r} \right)^{-1},
\end{align*}
and we write $f= g + b$, $b=\sum_{j} b_j$ with
\begin{equation*}
b_j(x) = \left( f(x) - f_{Q_j} \right)  \chi_{Q_j}(x),
\quad
f_{Q_j}= \frac{1}{\abs{Q_j}}
\int_{Q_j} f(x)\, dx.
\end{equation*}
For $g$, we have
\begin{equation*}
\norm{g}_{L^{p_1}}^{p_1}
\le \norm{g}_{L^\infty}^{p_1 - 1}\norm{g}_{L^1}
\lesssim  \left(A \lambda^{r}\right)^{p_1 -1}.
\end{equation*}
Thus \eqref{equ:WeakLp1Bound} gives
\begin{align*}
&\abs{
\{x :\,\, \abs{Tg (x)} > \lambda\}
}
\le \left(
M_{1} \norm{g}_{L^{p_1}} \lambda^{-1}
\right)^{q_1}
\\
&\lesssim
\left(
M_1 (A \lambda^r)^{1-1/p_1} \lambda^{-1}
\right)^{q_1}
=
\left(M_1 A^{1-1/p_1}\right)^{q_1}
\lambda^{-r}.
\end{align*}
Each $b_j$ satisfies
\begin{equation*}
\mathrm{supp}\, b_j \subset Q_j,
\quad
\int b_j(x)\, dx=0,
\quad
\frac{1}{\abs{Q_j}}
\int_{Q_j} \abs{b_j (x)}\, dx \lesssim
A \lambda^{r},
\end{equation*}
and thus
$\abs{Q_j}^{-1/p_0}
\left(A \lambda^{r}\right)^{-1} b_j $ is a constant multiple of
an $L^1$-atom for $H^{p_0}$
since $n/(n+1)< p_0<1$.
Hence we have
\begin{equation*}
\norm{b}_{H^{p_0}} ^{p_0}
\lesssim
\sum_{j}
\left(
\abs{Q_j}^{1/p_0}
A \lambda^{r}
\right)^{p_0}
\le
\left(
A \lambda^{r}
\right)^{p_0}
\left(A\lambda^{r}\right)^{-1}
= \left(
A \lambda^{r}
\right)^{p_0-1}.
\end{equation*}
Thus \eqref{equ:WeakHp0Bound} gives
\begin{align*}
&\abs{
\{x :\,\, \abs{Tb (x)} > \lambda\}
}
\le \left(
M_{0} \norm{b}_{H^{p_0}} \lambda^{-1}
\right)^{q_0}
\\
&\lesssim
\left(M_0 \left(A \lambda^{r} \right)^{1-1/p_0}
\lambda^{-1}
\right)^{q_0}
=\left(M_0 A^{1-1/p_0} \right)^{q_0}
\lambda^{-r}.
\end{align*}
Thus for $Tf =Tg + Tb$, combining the above estimates, we obtain
\begin{equation*}
\abs{
\{x:\,\, \abs{Tf (x)} > 2 \lambda\}
}
\lesssim
\left\{
\left(M_0 A^{1-1/p_0} \right)^{q_0}
+ \left(M_1 A^{1-1/p_1} \right)^{q_1}
\right\}
\lambda^{-r}.
\end{equation*}
We choose $A$ so that it minimizes the last expression,
and we obtain
\begin{equation*}
\abs{
\{x :\,\, \abs{Tf (x)} > 2 \lambda\}
}
\lesssim
\left(
M_0^{1-\theta} M_{1}^{\theta}
\lambda^{-1}
\right)^{r}.
\end{equation*}
This completes the proof of Lemma \ref{lem:Stein}.
\end{proof}

\begin{proof}[Proof of Theorem \ref{thm:WeakEstimate}]
Suppose $s_1, \dots, s_m$ and $p_1, \dots, p_m$ satisfy the assumptions of the theorem
and suppose for example $p_1=1$.
If we take $\epsilon>0$ sufficiently small,
then $s_1, \dots, s_m$ also satisfy the assumptions of the theorem
with $p_1=1$ replaced by $1 \pm \epsilon$.
Thus Theorem \ref{thm:General} yields
two estimates
\begin{align*}
&\norm{T_{\sigma}(f_1, f_2, \dots, f_m)
}_{L^{(p_{-},\infty)}}
\lesssim
A \norm{f_1}_{H^{1-\epsilon}}\norm{f_2}_{H^{p_2}}\cdots \norm{f_m}_{H^{p_m}},
\\
&\norm{T_{\sigma}(f_1, f_2, \dots, f_m)
}_{L^{(p_{+},\infty)}}
\lesssim
A \norm{f_1}_{L^{1+\epsilon}}\norm{f_2}_{H^{p_2}}\cdots \norm{f_m}_{H^{p_m}},
\end{align*}
where
$A=\sup_{j\in \mathbb Z} \norm{\sigma (2^j \cdot)\widehat{\psi}}
_{W^{(s_1, \dots, s_m)}}$ and
$p_{\pm}$ is given by
$1/(1 \pm \epsilon) + 1/p_2 + \dots + 1/p_m =1/p_{\pm}$.
We freeze the functions $f_2, \dots, f_m$ and
apply Lemma \ref{lem:Stein} to the linear operator
$f_{1}\mapsto T_{\sigma}(f_1, f_2, \dots, f_m)$ to obtain
\begin{equation*}
\norm{
T_{\sigma}(f_1, f_2, \dots, f_m)
}_{  L^{(p,\infty)}   }
\lesssim
A \norm{f_1}_{L^{1}} \norm{f_2}_{H^{p_2}} \cdots \norm{f_m}_{H^{p_m}}.
\end{equation*}
Repeated application of the same argument gives the desired weak type estimate.

 The necessity of the conditions
$s_i \ge n/2$ and \eqref{equ:SIndexCondGe} can be shown
by the same method as in \cite[Theorem 5.1]{GraHan}.
This completes the proof of Theorem \ref{thm:WeakEstimate}.
\end{proof}
\bibliographystyle{amsplain}

\end{document}